\documentclass{kepart2010}
\usepackage{amsmath,amsthm,amsfonts,amssymb,amscd,verbatim}
\usepackage[all,dvips]{xy}
\usepackage{epsfig,graphicx,verbatim,pifont}
\usepackage[usenames,dvips]{color}
\def\({\left(}
\def\){\right)}
\def\be {\begin{equation}}
\def\en{\end{equation}}

\def\F{\mathcal F (\ca)}
 \DeclareMathOperator{\C}{{\mathbb C}}
 
 \DeclareMathOperator{\R}{\mathbb R}
 \DeclareMathOperator{\A}{\ca^{*}}
 \DeclareMathOperator{\Z}{\mathbb Z}

 \DeclareMathOperator{\supp}{\text supp}

\def\ca{{\mathcal A}}
\def\cb{{\mathcal B}}
\def\cc{{\mathcal C}}

\def\cf{{\mathcal F}}

\def\cl{{\mathcal L}}
\def\cm{{\mathcal M}}

\def\cu{{\mathcal U}}
\def\cv{{\mathcal V}}

\def\Br{\mathbb R}
\def\Bz{\mathbb Z}
\def\Bn{\mathbb N}

\def\one{\pmb{1}}

\theoremstyle{plain}   
\begingroup 
\newtheorem{thm}{Theorem}[section]   
\newtheorem*{thm*}{Theorem}          
\newtheorem*{cor*}{Corollary}        
\newtheorem{cor}[thm]{Corollary}     

\newtheorem{lem}[thm]{Lemma}         
\newtheorem{prop}[thm]{Proposition}  
\newtheorem{rema}[thm]{Remark}
\endgroup

\theoremstyle{definition}
\newtheorem{defn}[thm]{Definition}
\newtheorem*{rem*}{Remark}
\newtheorem*{ack*}{Acknowledgment}
\newtheorem{prob}[thm]{Problem}

\theoremstyle{remark}
\newtheorem{rem}[thm]{Remark}        %
\newtheorem{ex}[thm]{Example}        

\theoremstyle{definition}

\numberwithin{equation}{section}

\begin{document}
\thispagestyle{empty}

\title[Symbolic Dynamics Viewpoint]
{Hidden Markov Processes in the Context of Symbolic Dynamics}





\author{Mike Boyle}
\address{Department of Mathematics, University of Maryland, College Park, MD 20742-4015 USA}
\email{mmb@math.umd.edu}

\author{Karl Petersen}
\address{Department of Mathematics,
CB 3250, Phillips Hall,
         University of North Carolina,
Chapel Hill, NC 27599 USA} \email{petersen@math.unc.edu}




\date{\today}
\begin{abstract}
{
 In an effort to aid communication among
different fields and perhaps facilitate progress on problems common
to all of them, this article discusses hidden Markov processes from
several viewpoints, especially that of symbolic dynamics, where they
are known as sofic measures, or continuous shift-commuting images of
Markov measures. It provides background, describes known tools and
methods, surveys some of the literature, and proposes several open
problems. }
\end{abstract}

\maketitle

\tableofcontents


\nocite{BoyleTuncel1984,BurkeRosenblatt1958,
  Kleene1956,Schutzenberger1961,
  BerstelReutenauer1988,HanselPerrin1989,
  BinkowskaKaminski1984,Furstenberg1960,
  MarcusPetersenWilliams1984,
  Blackwell1957,Billingsley1995,
  ChazottesUgalde2003,
  DownarowiczMauldin2005,
  Ellis,
  LedrappierWalters1977,
  Parry1964,ParryTuncel1982,
  Petersen1989,Petersen1998,
  PetersenShin2005,PetersenQuasShin2003,
  Phelps2002,ShannonWeaver1949,
  Shin2001,Shin2001-2,Shin2006,
  Walters1986,Kitchens1982}

  \nocite{
  Abdel-MoneimLeysieffer1982,
  Abdel-MoneimLeysieffer1984,
  Ahmad1977,
  Arbib1967,
  Bancilhon1974,
  BlackwellKoopmans1957,
  Bosch1974/75,
  Boudreau1968,
  BurkeRosenblatt1958-2,
  Ellis1976,
  EphraimMerhav2002,
  Erickson1970,
  FoxRubin1968,
  FoxRubin1969,
  FoxRubin1970,
  HachigianRosenblatt1962,
  Heller1967,
  Holland1968,
  InagakiFutumuraMutuura1972,
  Kelly1982,
  KomotaKimura1978, KomotaKimura1978-2,
  KomotaKimura1981,
  Leysieffer1967,
  Madsen1975,
  NeuhoffShields1982,
  Paz1975,
  Robertson1973,
  Robertson1973-2,
  Silio1979}

\section{Introduction}\label{sec_intro}

Symbolic dynamics is the study of shift (and other) transformations
on spaces of infinite sequences or arrays of symbols and maps
between such systems. A symbolic dynamical system, with a
shift-invariant measure, corresponds to a stationary stochastic
process. In the setting of information theory, such a system amounts
to a collection of messages. Markov measures and hidden Markov
measures, also called sofic measures, on symbolic dynamical systems
have the desirable property of being determined by a finite set of
data. But not all of their properties, for example the entropy, can
be determined by finite algorithms. This article surveys some of the
known and unknown properties of hidden Markov measures that are of
special interest from the viewpoint of symbolic dynamics. To keep
the article self contained, necessary background and related
concepts are reviewed briefly. More can be found in
\cite{LindMarcus1995, Petersen1989, Petersen1998, Walters1982}.

  We discuss methods and tools that have been useful in the study of
 symbolic systems, measures supported on them, and maps between
 them.
 Throughout we state several problems that we believe to be
 open and
meaningful for further progress.
 We review {a swath of }the complicated
literature starting around 1960 that deals with the problem of
recognizing hidden Markov measures, as
 closely related ideas were repeatedly rediscovered in varying
 settings and with varying degrees of generality or practicality.
 {Our focus is on the probability papers that relate most closely to
symbolic dynamics. We have left out much of the literature
concerning probabilistic and linear automata and control, but we
have tried to include the main ideas relevant to our problems.
  Some of the explanations that we give and connections that we
draw are new, as
 are some results near the end of the article.
 In Section \ref{sec_orders} we give bounds on the possible order (memory)
 if a given sofic measure is in fact a Markov measure,
  with the consequence that in some situations there is
  an algorithm for determining whether a hidden Markov measure is Markov.
 In Section \ref{sec_factorhiddenmarkovian} we show that every
 factor map is hidden Markovian, in the sense that every
 hidden Markov measure on an irreducible
 sofic subshift lifts to a fully supported hidden Markov measure.
}

\section{Subshift background}\label{sec_sftbackground}

\subsection{Subshifts}\label{sec_sfts}

Let $\ca$ be a set, usually finite or sometimes countable, which we
consider to be an alphabet of symbols.
 \be \ca^*=\bigcup_{k=0}^\infty \ca^k
 \en
denotes the set of all finite blocks or words with entries from
$\ca$, including the empty word, $\epsilon$;
 $\ca^+$ denotes the set of all nonempty words in $\ca^*$;
$\mathbb Z$ denotes the integers and $\mathbb Z_+$ denotes the
nonnegative integers.
 Let $\Omega(\ca)=\ca^\Bz$ and $\Omega^+(\ca)=\ca^{\Bz_+}$ denote the
set of all two or one-sided sequences with entries from $\ca$. If
$\ca=\{ 0,1,\dots ,d-1\}$ for some integer $d>1$, we denote
$\Omega(\ca)$ by $\Omega_d$ and $\Omega^+(\ca)$ by $\Omega^+_d$.
Each of these spaces is a metric space with respect to the metric
defined by setting for $x\neq y$
 \be k(x,y)=\min\{|j|:x_j \neq y_j\} \quad\text{and  }
 d(x,y)= e^{-k(x,y)}.
 \en
 For $i \leq j$ and $x \in \Omega(\ca)$
we denote by $x [i,j]$ the block or word $x_ix_{i+1}\dots x_j$.  If
$\omega=\omega_0\dots \omega_{n-1}$ is a block of length $n$, we
define
 \be \cc_0(w)=\{y \in \Omega(\ca):y[0,n-1]=\omega\} ,
 \en
 and, for $i \in \Bz$,
 \be
 \cc_i(\omega)=\{y \in \Omega(\ca):y[i,i+n-1]=\omega\} .
\en The cylinder sets $\cc_i(\omega),\omega \in \ca^*, i\in \mathbb
Z$,
are open and
closed and form a base for the topology of $\Omega(\ca)$.

In this paper, a {\em topological dynamical system} is
a continuous self map of a compact metrizable space.
The {\em shift transformation} $\sigma : \Omega_d \to \Omega_d$ is
defined by $(\sigma x)_i=x_{i+1}$ for all $i$. On $\Omega_d$ the
maps $\sigma$ and $\sigma^{-1}$ are one-to-one, onto, and
continuous. The pair $(\Omega_d, \sigma)$ forms a  topological
dynamical system  which is called the {\em full $d$-shift}.

If $X$ is a closed $\sigma$-invariant subset of $\Omega_d$, then the
topological dynamical system $(X,\sigma)$ is called a {\em
subshift}. In this paper, with ``$\sigma$-invariant'' we include
the requirement that the restriction of the shift be surjective.
Sometimes we denote a subshift $(X,\sigma)$ by only $X$,
the shift map being understood implicitly. When dealing with several
subshifts, their possibly different alphabets will be denoted by
$\ca(X), \ca(Y),$ etc.

The {\em language} $\cl(X)$ of the subshift $X$ is the set of all
finite words or blocks that occur as consecutive strings
 \be
 x[i,i+k-1]=x_ix_{i+1}\dots x_{i+k-1}
 \en
 in the infinite sequences $x$
which comprise $X$. Denote by $|w|$ the length of a string $w$. Then
 \be
 \cl(X) = \{w \in \ca^* : \text{there are  } n \in \Bz, y \in X
 \text{ such that } w=y_n \dots y_{n+|w|-1} \}.
 \en

Languages of (two-sided) subshifts are characterized by being {\em
extractive} (or {\em factorial}) (which means that every subword of
any word in the language is also in the language) and {\em
insertive} (or {\em extendable}) (which means that every word in the
language extends on both sides to a longer word in the language).

 For each subshift
$(X,\sigma)$ of $(\Omega_d,\sigma)$ there is a set $\cf(X)$ of
finite ``forbidden'' words such that \be X=\{x \in \Omega_d :
\text{for each  } i \leq j, x_ix_{i+1}\dots x_j \notin \cf(X)\} .
\en A {\em shift of finite type (SFT)} is a subshift $(X,\sigma)$ of
some $(\Omega (\ca),\sigma)$ for which it is possible to choose the
set $\cf(X)$ of forbidden words defining $X$ to be finite. (The
choice of set $\cf(X)$ is not uniquely determined.) The SFT is {\it
n-step} if it is possible to choose the set of words in $\cf(X)$ to
have length at most $n+1$.
  {
  We will sometimes use ``SFT" as an
adjective describing a dynamical system.}

One-step shifts of finite type may be defined by $0,1$ transition
matrices. Let $M$ be a $d \times d$ matrix with rows and columns
indexed by $\ca=\{0,1,\dots,d-1\}$ and entries from $\{0,1\}$.
Define
 \be\label{sft}
 \Omega_M=\{\omega \in \ca^\Bz: \text{for all  } n\in \Bz,
 M(\omega_n,\omega_{n+1})=1\}.
 \en
 These were called {\em topological
Markov chains} by Parry \cite{Parry1964}. A topological Markov
chain  $ \Omega_M$ may be viewed as a {\it vertex shift}:
its alphabet may be identified with the vertex set of a
 finite directed graph such that  there is an edge from
vertex $i$ to vertex $j$ if and only if $M(i,j)=1$. (A square matrix
with nonnegative integer entries can similarly be viewed as defining
an {\it edge shift}, but we will not need edge shifts in this
paper.) A topological Markov chain with transition matrix $M$ as
above is called {\em irreducible} if for all $i,j \in \ca$ there is
$k$ such that $M^k(i,j)>0$. Irreducibility corresponds to the
associated graph being strongly connected.

\subsection{Sliding block codes}\label{sec_codes}

Let $(X,\sigma)$ and $(Y,\sigma)$ be subshifts on alphabets
$\ca,\ca'$, respectively. For $k \in \Bn$, a {\em $k$-block code} is
a map $\pi: X \to Y$ for which there are $m,n \geq 0$ with $k=m+n+1$
and a function $\pi:\ca^k \to \ca'$ such that \be (\pi x)_i =
\pi(x_{i-m} \dots x_i \dots x_{i+n}). \en We will say that $\pi$ is
a {\em block code} if it is a $k$-block code for some $k$.

\begin{thm}[Curtis-Hedlund-Lyndon Theorem] For subshifts $(X,\sigma)$ and
$(Y,\sigma)$, a map $\psi : X \to Y$ is continuous and commutes with
the shift
($\psi\sigma=\sigma\psi$) if and only if it is a block code.
\end{thm}

If $(X,T)$ and $(Y,S)$ are topological dynamical systems,
 then a {\em factor map} is a continuous onto map $\pi :X \to Y$
such that $\pi T=S \pi$. $(Y,S)$ is called a {\em factor} of
$(X,T)$, and $(X,T)$ is called an {\em extension} of $(Y,S)$. A
one-to-one factor map is called an {\em isomorphism} or {\em
topological conjugacy}.

Given  a subshift $(X,\sigma)$,
$r\in \mathbb Z$ and $k \in \mathbb Z_+ $,
there is
a block code $\pi=\pi_{r,k}$ onto
the subshift which is
the
{\it k-block presentation} of $(X,\sigma )$,
 by the rule
\be (\pi x)_i = x[i+r,i+r+1,\dots ,i+r+k-1] \quad\text{for all } x
\in X. \en Here  $\pi$ is a topological conjugacy between
$(X,\sigma)$ and its image $(X^{[k]},\sigma)$ which is a subshift of
the full shift on the alphabet $\ca^k$.

Two factor maps $\phi, \psi$ are {\it topologically equivalent} if
there exist topological conjugacies $\alpha, \beta$ such that
$\alpha \phi \beta = \psi $. In particular, if $\phi$ is a block
code with $(\phi x)_0$ determined by $x[-m,n]$ and $k=m+n+1$ and
$\psi$ is the composition $(\pi_{m,k})^{-1}$ followed by $\phi$,
then
$\psi$ is a 1-block code (i.e. $(\psi x)_0 = \psi (x_0)$) which is
topologically equivalent to $\phi$.

A {\em sofic} shift is a subshift which is the image of a shift of
finite type under a factor map.
A sofic shift $Y$ is {\em irreducible} if it is the image of an
irreducible shift of finite type under a factor map.
(Equivalently, $Y$ contains a point with a dense forward orbit.
Equivalently, $Y$ contains a point with a dense orbit, and the
periodic points of $Y$ are dense.)

\subsection{Measures}\label{sec_measures}

Given a subshift $(X,\sigma)$, we denote by $\cm (X)$ the set of
$\sigma$-invariant Borel probability measures on $X$. These are the
measures for which the coordinate projections $\pi_n(x)=x_n$ for $x
\in X, n \in \Bz$, form a two-sided finite-state stationary
stochastic process.

Let $P$ be a {}{ $d \times d$  stochastic} matrix and $p$ {}{a}
 stochastic row vector such that $pP=p$.
{}{(If $P$ is irreducible, then $p$ is unique.)} Define a $d \times
d$ matrix $M$ with entries from $\{ 0,1\}$ by $M(i,j)=1$ if and only
if $P(i,j)>0$. Then $P$ determines a 1-step stationary
($\sigma$-invariant) Markov measure $\mu$ on the shift of finite
type $\Omega_M$ by
 \be
 \begin{aligned}
 \mu(\cc_i(\omega[i,j]))&=\mu\{y \in \Omega_M:y[i,j]=\omega_i\omega_{i+1}\dots \omega_j\}\\
 &=p(\omega_i)P(\omega_i,\omega_{i+1}) \cdots P(\omega_{j-1},\omega_j)
 \end{aligned}
 \en
 (by the Kolmogorov Extension Theorem).

For $k \geq 1$, we say that a measure $\mu \in \cm(X)$ is
{\em
$k$-step Markov}
(or more simply
{\em k-Markov})
 if for all $i \geq 0$ and all $j \geq k-1$
and all $x$ in $X$,
 \be
\mu(\cc_0(x[0,i])|\cc_0(x[-j,-1]))=\mu(\cc_0(x[0,i])|\cc_0(x[-k,-1])).
\en A measure is $1$-step Markov if and only if it is determined by
a {}{ pair $(p,P)$ } as above. A measure is $k$-step Markov if and
only if its image under the topological conjugacy taking
$(X,\sigma)$ to its $k$-block presentation is 1-step Markov. We say
that a measure is {\em Markov} if it is $k$-step Markov for some
$k$. The set of $k$-step Markov measures is denoted by $\cm_k$
(adding an optional argument to specify the system or transformation
if necessary.) {\em From here on, ``Markov" means ``shift-invariant
Markov with full support"}, that is, every nonempty cylinder subset
of $X$ has positive measure. {}{With this convention, a Markov
measure with defining matrix $P$ is ergodic if and only if $P$ is
irreducible.}

A probabilist might ask for  motivation for bringing  in the
machinery of topological and dynamical systems when we want to study
a  stationary stochastic process. First, looking at $\cm(X)$ allows
us to consider and compare many measures in a common setting. By
relating them to continuous functions (``thermodynamics''---see
Section \ref{thermo} below) we may find some distinguished measures,
for example maximal ones in terms of some variational problem.
Second, by topological conjugacy we might be able to simplify a
situation conceptually; for example,  many problems involving block
codes reduce to problems involving just 1-block codes. And third,
with topological and dynamical ideas we might see (and know to look
for) some structure or common features, such as invariants of
topological conjugacy, behind the complications of a particular
example.

\subsection{Hidden Markov (sofic) measures}\label{sec_hmms}

If $(X,\sigma)$ and $(Y,\sigma)$ are subshifts and $\pi: X \to Y$ is
a sliding block code (factor map), then each measure $\mu \in
\cm(X)$ determines a measure $\pi \mu \in \cm(Y)$ by
 \be (\pi\mu)(E) =\mu(\pi^{-1}E) \quad\text{for each measurable } E \subset Y .
 \en
(Some authors write $\pi_*\mu$
or  $\mu\pi^{-1}$
for $\pi\mu$.)

 If $X$ is SFT, $\mu$ is a Markov measure on $X$ and $\pi: X \to
Y$ is a sliding block code, then $\pi\mu$ on $Y$ is called a {\em
hidden Markov measure}
 or {\em sofic measure}.
(Various other names, such as ``submarkov" and ``function of a
Markov chain" have also been used for such a measure or the
 associated stochastic process.)
{}{Thus $\pi \mu$ is a convex combination of images of ergodic
Markov measures.
 {\em From here on, unless
otherwise indicated, the
domain of a Markov measure is assumed to be an irreducible SFT, and
the Markov measure is assumed to have full support (and thus by
irreducibility be ergodic).
 Likewise, unless otherwise indicated, a
sofic measure is assumed to have full support and to be the image of
an ergodic Markov measure.}} Then the sofic measure is ergodic and it
is defined on an irreducible sofic subshift.
 Hidden Markov measures provide a natural way to model
systems governed by chance in which dependence on the past of
probabilities of future events is limited (or at least decays, so
that approximation by Markov measures may be reasonable) and
complete knowledge of the state of the system may not be possible.

Hidden Markov processes are often defined as probabilistic functions
of Markov chains (see for example \cite{EphraimMerhav2002}), but by
enlarging the state space each such process can be represented as a
deterministic function of a Markov chain, such as we consider here
(see \cite{BaumPetrie1966}).

The definition of hidden Markov measure raises several questions.
\begin{prob}\label{prob_markovimage}
Let $\mu$ be a 1-step Markov measure on $(X,\sigma)$ and $\pi :X \to
Y$ a 1-block code. The image measure may not be Markov---see Example
\ref{figblack1/3}. What are necessary and sufficient conditions for
$\pi\mu$ to be 1-step Markov?
\end{prob}
This problem has been solved, in fact several times.
Similarly,
given $\mu$ and $\pi$, it is possible to determine whether $\pi\mu$
is $k$-step Markov. Further, given $\pi$ and a Markov measure $\mu$,
it is possible to specify $k$ such that either $\pi\mu$ is $k$-step
Markov or else is not Markov of any order. These results are
discussed in Section \ref{sec_markov}.

\begin{prob}\label{prob_identify}
Given a shift-invariant measure $\nu$ on $(Y,\sigma)$, how can one
tell whether or not $\nu$ is a hidden Markov measure? If it is, how
can one construct Markov measures of which it is the image?
\end{prob}
The answers to Problem \ref{prob_identify}
provided by various authors are discussed in Section
\ref{sec_ident}. The next problem reverses the viewpoint.

\begin{prob}\label{prob_lift}
Given a sliding block code $\pi:X \to Y$
and a Markov measure $\nu$ on $(Y,\sigma)$, does there exist a
Markov
measure $\mu$ on $X$ such that $\pi\mu = \nu$?
\end{prob}

In Section \ref{sec_mapsandthermo},
we take up Problem \ref{prob_lift} (which apart from special
cases remains open)
and some theoretical background that motivates
it.

Recall that a factor map
$\pi : X\to Y$
between irreducible sofic shifts
has a {\it degree}, which is the cardinality
of the preimage of any doubly transitive point of $Y$
 \cite{LindMarcus1995}. (If the cardinality is infinite,
it can only be the power of the continuum, and we simply
write $\textnormal{degree}(\pi )=\infty$.)
 If $\pi$
has degree $n<\infty$, then an ergodic measure $\nu$ with full
support on $Y$ can lift to  at most $n$  ergodic measures on $X$. We
say that the {\it degree of a hidden Markov measure} $\nu$, also
called its {\em sofic degree}, is the minimal degree of a factor map
which sends some Markov measure to $\nu$.

\begin{prob}\label{prob_soficdegree}
Given a hidden Markov measure $\nu$ on $(Y,\sigma)$, how can one
determine the degree of $\nu$? If the degree is $n< \infty$, how can
one construct Markov measures of which
$\nu$ is the image under a degree $n$ map?
\end{prob}

We conclude this section with  examples.

\begin{ex}\label{figshinnonmarkovian}

An example was given in \cite{MarcusPetersenWilliams1984}
 of a code $\pi : X \to Y$ that is {\em non-Markovian}: some Markov measure
 on $Y$ does not lift
 to any Markov measure on $X$, and hence (see Section \ref{sec_markovian})
 no Markov measure on $Y$ has a Markov preimage on $X$. The following diagram
 presents a simpler example, due to Sujin Shin \cite{Shin2001-2,Shin2006},
  of such a map. Here $\pi$ is a 1-block code: $\pi (1)=1$ and $\pi
 (j)=2$ if $j\neq 1$.

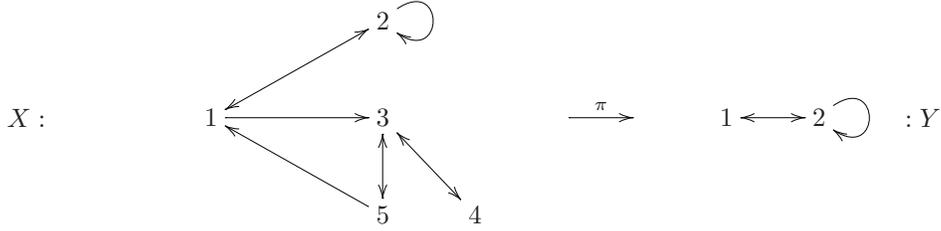
\begin{center}
\xymatrix{
& & &  & 2 \ar@(ur,dr)[] \\
X: & & 1\ar@{<->}[urr] \ar@{<-}[drr] \ar@{->}[rr] & & 3 \ar@{<->}[d]
\ar@{<->}[dr] & & \ar [r]^\pi & & 1
\ar@{<->} [r] & 2 \ar@(ur,dr)[] &  :Y\\
& & & &  5 & 4 }
\end{center}
\end{ex}

\begin{ex}\label{figblack1/2}
Consider the shifts of finite type given by the graphs below, the
1-block code $\pi$ given by the rule $\pi(a)=a,
\pi(b_1)=\pi(b_2)=b$, and the Markov measures $\mu, \nu$ defined by
the transition probabilities shown on the edges. We have $\pi\mu =
\nu$, so the code is {\em Markovian}---some Markov measure maps to a
Markov measure.
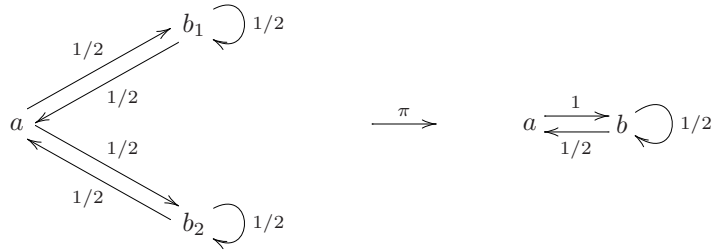
\begin{center}
\xymatrix{
& & & &     b_1 \ar@<.7ex>[dll]^{1/2}\ar@(ur,dr)[]^{1/2}  \\
& & a\ar@<.7ex>[urr]^{1/2} \ar@<.7ex>[drr] ^{1/2}& & & & \ar [r]^\pi
& & a
\ar@<.7ex>[r]^{1}& b \ar@<.7ex>[l]^{1/2}\ar@(ur,dr)^{1/2} \\
& & & &    b_2 \ar@<.7ex>[ull]^{1/2}\ar@(ur,dr)[] ^{1/2} }
\end{center}
\end{ex}

\begin{ex}\label{figblack1/3}
This example uses the same shifts of finite type and 1-block code as
in Example \ref{figblack1/2},
 but we define a new 1-step Markov measure on the
upstairs shift of finite type $X$ by assigning transition
probabilities as shown.

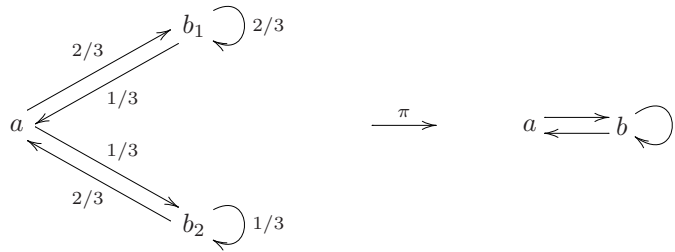
\begin{center}
\xymatrix{
& & & &     b_1 \ar@<.7ex>[dll]^{1/3}\ar@(ur,dr)[]^{2/3}  \\
& & a\ar@<.7ex>[urr]^{2/3} \ar@<.7ex>[drr] ^{1/3}& & & & \ar [r]^\pi
& & a
\ar@<.7ex>[r]& b \ar@<.7ex>[l]\ar@(ur,dr) \\
& & & &    b_2 \ar@<.7ex>[ull]^{2/3}\ar@(ur,dr)[] ^{1/3}  }
\end{center}
The entropy of the Markov measure $\mu$
(the definition is recalled in Sec. \ref{thermo})
 is readily obtained from the
familiar formula $-\sum p_i P_{ij}\log P_{ij}$, but
there is no such simple rule for computing the entropy of $\nu$.
 If $\nu$ were the finite-to-one image of some other
Markov measure $\mu'$, maybe on some other shift of finite type,
then we would have $h(\nu ) = h( \mu ') $ and the entropy of $\nu$
would be easily computed by applying the familiar formula to $\mu'$.
But for  this example (due to Blackwell \cite{Blackwell1957}) it can
be shown \cite{MarcusPetersenWilliams1984} that $\nu$ is not the
finite-to-one image of any Markov measure. Thus Problem
\ref{prob_soficdegree} is relevant to the much-studied problem of
estimating the entropy of a hidden Markov measure (see
\cite{HanMarcus2006,HanMarcus2007} and their references).
\end{ex}
\begin{ex}\label{ex_Walters}
In this example presented in \cite{Walters1986}, $X=Y=\Sigma_2=$
full $2$-shift, and the factor map is the 2-block code
 \be (\pi x)_0=x_0+x_1 \mod 2 .
 \en
Suppose $0<p<1$ and $\mu_p$ is the Bernoulli (product) measure on
$X$,
with $\mu (\mathcal C_0(1))=p$.
Let $\nu_p$ denote the hidden Markov measure
$\pi \mu_p=\pi \mu_{1-p}$.
If $p\neq 1/2$, then $\nu_p$ is
a hidden Markov measure strictly of degree 2 (it is not degree 1).

\end{ex}

\section{Factor maps and thermodynamical
concepts}\label{sec_mapsandthermo}

\subsection{Markovian and non-Markovian maps}\label{sec_markovian}

We have mentioned (Example \ref{figblack1/3}) that the image under a
factor map $\pi: X \to Y$ of a Markov measure need not be Markov,
and (Example \ref{figshinnonmarkovian}) that a Markov measure on $Y$
need not have any Markov preimages. In this section we study maps
that do not have the latter undesirable property.
Recall our convention: a Markov measure
is required to have full support.
\begin{defn}\cite{BoyleTuncel1984}
A factor map $\pi : \Omega_A \to \Omega_B$ between irreducible
shifts of finite type ($A$ and $B$ are $0,1$ transition matrices,
see (\ref{sft})) is {\em Markovian} if for every  Markov measure
$\nu$ on $\Omega_B$, there is a Markov measure on $\Omega_A$ such
that $\pi \mu = \nu$.
\end{defn}
\begin{thm}\label{th_markovian}\cite{BoyleTuncel1984} For a factor map $\pi: \Omega_A \to
  \Omega_B$ between irreducible
shifts of finite type, if there exist
any fully supported Markov $\mu$ and $\nu$ with $\pi \mu = \nu$,
then  $\pi$ is Markovian.
\end{thm}

Note that if a factor map is Markovian, then so too is every factor
map which is topologically equivalent to it, because a topological
conjugacy takes Markov  measures to Markov measures. We will see a
large supply of Markovian maps (the ``e-resolving factor maps'') in
Section  \ref{sec_resolving}.

These considerations lead to a reformulation of Problem
\ref{prob_lift}:
\begin{prob}\label{prob_markovianmap}
Give a procedure to decide, given a factor map $\pi: \Omega_A \to
\Omega_B$, whether $\pi$ is Markovian.
\end{prob}

We sketch the proof of Theorem \ref{th_markovian} for the 1-step
Markov case: if any 1-step Markov measure on $\Omega_B$ lifts to a
1-step Markov measure, then every 1-step Markov measure on
$\Omega_B$ lifts to a 1-step Markov measure. For this, recall that
if $M$ is an irreducible matrix with spectral radius $\rho$, with
positive right eigenvector $r$, then the {\em stochasticization} of
$M$ is the stochastic matrix \be \label{stochasticization}
\text{stoch}(M)= \frac 1\rho D^{-1}MD \ , \en where $D$ is the
diagonal matrix with diagonal entries $D(i,i)=r(i)$.

Now suppose that $\pi:\Omega_A\to \Omega_B$ is a 1-block factor map,
with $\pi (i)$ denoted $\overline i$ for all $i$ in the alphabet of
$\Omega_A$; that $\mu ,\nu$ are 1-step Markov measures defined by
stochastic matrices $P,Q$; and that $\pi\mu = \nu$.
Suppose that $\nu'\in \mathcal M(\Omega_B)$ is defined by a
stochastic matrix $Q'$. We will find a stochastic matrix $P'$
defining $\mu '$ in $\mathcal M(\Omega_A)$
such that $\pi \mu'=\nu'$.

First define a matrix $M$ of size matching $P$ by
$M(i,j)=0$ if $P(i,j)=0$ and
otherwise
 \be  \label{markovmarkovian}
M(i,j)=
Q'(\overline i, \overline j) P(i,j)/ Q(\overline i, \overline j) .
 \en
This matrix $M$ will have spectral radius 1.
Now set $P'=\text{stoch}(M)$.
The proof that $\pi \mu' = \nu'$ is a straightforward computation
that  $\pi \mu' = \nu'$ on  cylinders
$\mathcal C_0(y[0,n])$
 for all $n\in \mathbb N$ and $y\in \Omega_B$.
This construction is the
germ of a more general thermodynamic result, the background for
which we develop in the next section. We finish this section
with an example.

\begin{ex}\label{exliftone}
In this example one sees explicitly how being able to lift one
Markov measure to a Markov measure, allows one to lift other Markov
measures to Markov measures.

Consider the 1-block code $\pi$ from $\Omega_3=\{0,1,2\}^{\mathbb
Z}$ to $\Omega_2=\{0,1\}^{\mathbb Z}$, via $0\mapsto 0$ and
$1,2\mapsto 1$. Let  $\nu$ be the  1-step Markov measure on
$\Omega_2$
 given by
the transition matrix
\begin{equation*}
\begin{pmatrix}
 1/2 &  1/2 \\
 1/2 & 1/2
\end{pmatrix} \ .
\end{equation*}
Given positive numbers $\alpha , \beta, \gamma <1 $, the stochastic
matrix \be
\begin{pmatrix}
1/2 &\alpha (1/2) & (1-\alpha ) (1/2)   \\
1/2 &\beta (1/2) & (1-\beta ) (1/2)     \\
1/2 & \gamma (1/2) & (1-\gamma ) (1/2)
\end{pmatrix}
\en defines a 1-step Markov measure on $\Omega_3$ which
$\pi$ sends to $\nu$.

Now, if $\nu'$ is  any other 1-step Markov measure  on $X_2$, given
by a stochastic matrix
\begin{equation*}
\begin{pmatrix}
p & q \\
r & s
\end{pmatrix} \  ,
\end{equation*}
then $\nu'$ will lift to the 1-step Markov measure defined by the
stochastic matrix \be
\begin{pmatrix} \label{resolvingsplit}
p &\alpha  q & (1-\alpha )  q   \\
r &\beta s & (1-\beta ) s  \\
r & \gamma s & (1-\gamma ) s
\end{pmatrix} \ .
\en
\end{ex}

\subsection{Thermodynamics on subshifts 001}\label{thermo}
We recall the definitions of entropy and pressure and how the
thermodynamical approach provides convenient machinery for dealing
with Markov measures (and hence eventually, it is hoped, with hidden
Markov measures).

Let $(X,\sigma )$ be a subshift and $\mu \in \mathcal M(X)$ a
shift-invariant Borel probability measure on $X$. The {\em
topological entropy} of $(X,\sigma)$ is
 \be h(X)= \lim_{n\to\infty} \frac1n\log|\{x[0,n-1]:x\in X\}|.
 \en
 The {\em measure-theoretic entropy}
of the measure-preserving system $(X,\sigma,\mu)$ is
 \be h(\mu)=h_{\mu}(X) =\lim_{n\to\infty} \frac 1n
 \sum \{-\mu (\cc_0(w))\log \mu (\cc_0(w)):w \in
 \{x[0,n-1]: x\in X\}\} .
 \en
 (For more background on these concepts,
one could consult \cite{Petersen1989,Walters1982}.)

{\em Pressure} is a refinement of entropy which takes into account
not only the map $\sigma: X\to X$ but also weights coming from a
given ``potential function" $f$ on $X$. Given a continuous
real-valued function $f\in C(X, \mathbb R)$, we define the {\em
pressure of $f$ (with respect to $\sigma$)} to be
 \be
 P(f,\sigma) = \lim_{n\to\infty} \frac 1n \log \sum
 \{\exp[S_n(f,w)]:w
\in \{x[0,n-1]: x\in X\}\},
\en
where
 \be
 S_n(f,w) =\sum_{i=0}^{n-1} f(\sigma^i x)
 \quad\text{for some }x\in X
 \quad\text{such that } x[0,n-1]=w .
 \en
 (In the limit the choice of
$x$ doesn't matter.) Thus,
 \be \text{if } f \equiv 0,
\ \text{then } P(f,\sigma )=
h(X) .
\en

The pressure functional satisfies the important {\em
Variational Principle}:
 \be P(f,\sigma ) = \sup \{ h(\mu) + \int f
\ d\mu : \mu \in \mathcal M(X)\} .
 \en
 An {\em equilibrium state}
for $f$ (with respect to $\sigma$) is a measure $\mu = \mu_f$ such
that
 \be P(f,\sigma) = h(\mu) + \int f \ d\mu .
  \en

{ Often (e.g., when the potential function $f$ is H\"older
continuous on an irreducible shift of finite type), there is a
unique equilibrium state $\mu_f$ which is a {\em (Bowen) Gibbs
measure} for $f$: i.e.,
 $P(f,\sigma )= \log (\rho )$, and
 \be \label{Eq_BowenGibbs}
 \mu_f (\cc_0(x[0,n-1])) \sim
\rho^{-n}\exp S_nf(x) . \en
Here ``$\sim$'' means the ratio of the
two sides is bounded above and away from zero,
 uniformly in $x$ and $n$.
}

If $f\in C(\Omega_A, \mathbb R)$, depends on only two coordinates,
$f(x) = f(x_0x_1)$ for all $x \in \Omega_A$, then $f$ has a unique
equilibrium state $\mu_f$, and $\mu_f \in \mathcal M( \Omega_A)$.
This measure $\mu_f$ is the 1-step Markov measure defined by the
stochastic matrix $P= \textnormal{stoch}(Q)$, where
\be Q(i,j)=
\begin{cases}
0\qquad&\text{ if } A(i,j)=0, \\
\exp[f(ij)] &\text{ otherwise } \ .
\end{cases}
\en (For an exposition see \cite{ParryTuncel1982}.)

 The pressure of $f$ is $\log \rho$,
where $\rho $ is the spectral radius of $Q$.
Conversely, a Markov measure with stochastic transition
matrix $P$ is the equilibrium state of the potential function
$f[ij]=\log P(i,j)$.

By passage to the
$k$-block presentation, we can generalize to the
case of $k$-step Markov measures: if $f(x) = f(x_0x_1\cdots
x_k)$, then $f$ has a unique equilibrium state $\mu$, and
$\mu$ is a $k$-step Markov measure.

\begin{defn}
We say that a  function on a subshift $X$  is {\em locally
constant} if there is $m \in \Bn$ such that $f(x)$ depends only on
$x[-m,m]$. LC($X,\mathbb R$) is the vector space of
locally constant real-valued
functions on $X$. $C_k(X,\mathbb R)$ is the set of $f$
in LC($X,\mathbb R$) such that $f(x)$ is determined by
$x[0,k-1]$.
\end{defn}

We can now express a viewpoint on Markov measures, due to Parry and
Tuncel \cite{Tuncel1981,ParryTuncel1982-2}, which follows from the
previous results.

\begin{thm}\cite{ParryTuncel1982-2}
Suppose
 $\Omega_A$ is  an irreducible shift of finite type;
 $k \geq 1$;  and $f,g \in C_k(X,\mathbb R)$.
Then the following are equivalent.
\begin{enumerate}
\item $\mu_f = \mu_g$.
\item There are $h\in C(X, \mathbb R)$ and $c \in \Br$ such that $f=g +
(h-h\circ \sigma) + c$.
\item There are $h\in C_{k-1}(X, \mathbb R )$ and $c \in \Br$ such that $f=g +
(h-h\circ \sigma) +c$.
\end{enumerate}
\end{thm}

\begin{prop}\cite{ParryTuncel1982-2} \label{markovasfunctions}
Suppose  $\Omega_A$ is an irreducible shift of finite type. Let \be
W=\{h-h\circ \sigma+ c: h \in \text{LC}(\Omega_A,\mathbb R ),
c\in\Br\}\ .
\en
Then the rule $[f]\mapsto \mu_f$ defines maps
\begin{align*}
C_k(\Omega_A,\mathbb R)/W \ &\to \ \mathcal M_k(\sigma_A) \  \\
\text{LC}(\Omega_A,\mathbb R)/W \ &\to \ \cup_k \mathcal
M_k(\sigma_A),
\end{align*}
and these maps are bijections.
\end{prop}

\subsection{Compensation functions}

Let $\pi : (X,T) \to (Y,S)$ be a factor map between topological
dynamical systems. A {\em compensation function} for the factor map
is a continuous function $\xi:X \to \mathbb R$ such that
\begin{equation}\label{compeq}
P_Y(V) = P_X(V \circ \pi + \xi) \quad\text{for all } V \in \cc
(Y,\Br) .
\end{equation}

Because $h(\pi\mu) \leq h(\mu)$ and $\int V\, d(\pi \mu)
 = \int V \circ \pi
\, d\mu$,
 we always have
\begin{align}
P_Y(V) &= \sup\{ h(\nu) + \int_Y V\,d\nu : \nu \in \cm (Y)
  \} \\
 &\leq  \sup\{ h(\mu) + \int_X V \circ \pi \,d\mu : \mu \in \cm (X)
   \}
= P_X(V \circ \pi) ,
   \end{align}
with possible strict inequality when $\pi $ is infinite-to-one, in
which case a strict inequality $h(\mu )> h(\pi \mu)$ can arise from
(informally) the extra information/complexity arising from motion in
fibers over points of $Y$. The pressure equality (\ref{compeq})
tells us that the addition of a compensation function $\xi$ to the
functions $V\circ \pi$ takes into account (and exactly cancels out),
for all potential functions $V$ on $Y$ at once, this measure of
extra complexity. Compensation functions were introduced in
\cite{BoyleTuncel1984} and studied systematically in
\cite{Walters1986}. A compensation function is a kind of oracle for
how entropy can appear in a fiber. The Markovian case is  the case
in which the oracle has finite range, that is, there is a locally
constant compensation function.

A compensation function for a factor map $\pi : X \to Y$ is {\it
saturated} if it has the form $G\circ \pi$ for a continuous function
$G$ on $Y$.

\begin{ex}\label{ex_compfn1}
For the factor map in Examples \ref{figblack1/2} and
\ref{figblack1/3}, the formula
 \be
 G(y)=\begin{cases}
-\log 2 &\text{ if } y=.a\dots\\
0 &\text{ if } y=.b\dots
\end{cases}
\en determines a saturated compensation function $G \circ \pi$ on
 $\Omega_A$.
 The  sum (or {\em cocycle}) $S_nG(y)=G(y)+G(\sigma y) + \dots +
G(\sigma^{n-1}y)$ measures the growth  of the number of preimages
of initial blocks of $y$: \be |\pi^{-1}(y_0\dots y_{n-1})| =
2^{\text{\#}\{i: y_i=a,0\leq i < n\} \pm 1 } \sim 2^{\text{\#}\{i:
y_i=a,0\leq i < n\} } = e^{-S_nG(y)}. \en
\end{ex}

\begin{ex}\label{ex_compfn2}
In the situation described at the end of Section
\ref{sec_markovian}, in which a 1-step Markov measure maps to a
1-step
Markov measure under a 1-block map, an associated compensation
function is \be \xi(x)= \log P(i,j) -\log Q(\overline i, \overline
j) \quad\text{when } x_0x_1=ij . \en
\end{ex}

\begin{thm}\cite{BoyleTuncel1984,Walters1986}\label{markoviancase}
Suppose that $\pi : \Omega_A \to \Omega_B$ is a factor map between
irreducible shifts of finite type, with $f\in \text{LC}(\Omega_A)$
and
$g\in \text{LC}(\Omega_B)$, and $\pi \mu_f = \mu_g$. Then
there is a constant $c$ such that
$f-g\circ \pi +c$ is a compensation function. Conversely, if
$\xi$ is a
locally constant compensation function, then
$\mu_{\xi+g\circ \pi} $ is Markov and
$\pi \mu_{\xi+g\circ \pi}  = \mu_g$.
\end{thm}

In Theorem \ref{markoviancase},
the locally constant compensation function $\xi$ relates potential
functions on $\Omega_B$ to their lifts by composition on $\Omega_A$
in the same way that the corresponding equilibrium states are
related: \be
\begin{gathered}
 \text{LC}(\Omega_B)\ \hookrightarrow \
\text{LC}(\Omega_A) \quad\text{via } g\to (g\circ \pi)+\xi\\
\cm(\Omega_B) \hookrightarrow \cm(\Omega_A) \quad\text{via } \mu_g
\to \mu_{(g\circ \pi)+\xi} .
\end{gathered}
\en

Theorem \ref{markoviancase} holds if we replace the class of locally
constant functions with the class of H\"{o}lder (exponentially
decaying) functions, or with functions in the larger and more
complicated ``Walters class'' (defined in \cite[Section
4]{Walters1986}). More generally, the arguments in  \cite[Theorem
4.1]{Walters1986} go through to prove the following.

\begin{thm}\label{walterishtheorem}
Suppose that $\pi : \Omega_A \to \Omega_B$ is a factor map between
irreducible shifts of finite type. Let $\mathcal V_A,\mathcal V_B$ be real
vector spaces of functions in $C(\Omega_A, \mathbb R),C(\Omega_B,
\mathbb R)$ respectively such that the following hold.
\begin{enumerate}
\item
$\mathcal V_A$ and $\mathcal V_B$ contain the locally constant
  functions.
\item
If $f$ is in $\mathcal V_A$ or $\mathcal V_B$,
then $f$ has a unique equilibrium
state $\mu_f$, and $\mu_f$ is a Gibbs measure.
\item
If $f\in \mathcal V_B$, then $f\circ \pi \in \mathcal V_A$.
\end{enumerate}
Suppose $f\in \mathcal V_A$ and
$g\in \mathcal V_B$, and $\pi \mu_f = \mu_g$. Then there is a
constant $C$ such that
$f-g\circ \pi +C$ is a compensation function. Conversely, if
$\xi$ in $\mathcal V_A$ is a
 compensation function,
then for all $g\in \mathcal V_B$ it holds that
$\pi \mu_{\xi+g\circ \pi}  = \mu_g$.

Moreover, if $G \in \mathcal V_B$, then $G\circ \pi$ is a
compensation function if and only if there is $c \geq 1$ such that
\begin{equation}
\frac{1}{c} \leq e^{S_nG(y)} \, |\pi^{-1}(y_0\dots y_{n-1})| \leq c
\text{ for all } y,n.
\end{equation}
\end{thm}

 {\begin{prob}\label{prob_wishingforcompfn} Determine whether there
exists a factor map $\pi :X \to Y$ between mixing SFT's and a
potential function $F \in \cc(X)$ which is {\em not} a compensation
function but has a unique equilibrium state $\mu_F$ whose image
$\pi\mu_F$ is the measure of maximal entropy on $Y$. If there were
such an example, it would show that the assumptions on function
classes in Theorem \ref{walterishtheorem} cannot simply be dropped.
\end{prob} }

We finish this section with some more general statements
about compensation functions for factor maps between shifts
of finite type.

\begin{prop}\label{prop_compfns} \cite{Walters1986}
Suppose that $\pi : \Omega_A \to \Omega_B$ is a factor map between
irreducible shifts of finite type. Then
\begin{enumerate}
\item There exists a compensation function.
\item If $\xi$ is a compensation function, $g\in \C(\Omega_B,\mathbb
  R),$ and $\mu$ is an equilibrium state of $\xi + g\circ \pi$,
then $\pi \mu$ is an equilibrium state of $g$.
\item The map $\pi$ takes the measure of maximal entropy (see Section
\ref{sec_relmaxent}) of $\Omega_A$ to that
of $\Omega_B$  if and only if there is a
{\em constant} compensation function.
\end{enumerate}
\end{prop}

Yuki Yayama \cite{Yayama2009} has begun the study of
compensation functions which are  bounded Borel functions.

\subsection{Relative pressure}\label{sec_relpressure}

When studying factor maps, relativized versions of entropy and
pressure are relevant concepts. Given a factor map $\pi: \Omega_A
\to \Omega_B$ between shifts of finite type, for each $n = 1, 2,
\cdots$ and $y \in Y$, let $D_n (y)$ be a set consisting of exactly
one point from each nonempty set $[x_0 \cdots x_{n-1}] \cap \pi^{-1}
(y)$. Let $V \in \cc(\Omega_A, \Br)$ be a potential function on
$\Omega_A$. For each $y \in \Omega_B$, the {\em relative pressure of
$V$ at $y$ with respect to $\pi$} is defined to be \be P(\pi, V)(y)
= \limsup_{n \rightarrow \infty} \frac{1}{n} \log \bigg[ \sum_{x \in
D_n (y)} \exp \Big( \sum_{i=0}^{n-1} V(\sigma^i x) \Big) \bigg]. \en
The {\em relative topological entropy function} is defined for all
$y \in Y$ by \be P (\pi, 0) (y) = \limsup\limits_{n \rightarrow
\infty} \frac{1}{n} \log \big\vert D_n (y) \big\vert, \en the
relative pressure of the potential function $V \equiv 0$.

For the relative pressure function, a {\em Relative Variational
Principle} was proved by Ledrappier and Walters
(\cite{LedrappierWalters1977}, see also
\cite{DownarowiczSerafin2002}):
 for all
$\nu$ in $\cm(\Omega_B)$ and all $V$ in $C(\Omega_A)$,
 \be
 \int\!P(\pi, V) \, d \nu = \sup\limits \Big\{ h(\mu) + \int\!V d
\mu : \pi\mu =\nu \Big\} - h(\nu). \en

In particular, for a fixed $\nu \in \cm(\Omega_B)$, the maximum
measure-theoretic entropy of a measure on $\Omega_A$ that maps under
$\pi$ to $\nu$ is given by
\begin{align}\label{eq_maxent}
h(\nu ) + \sup \{h_{\mu}(X|Y) :\pi\mu =\nu\}
&= h(\nu ) +\sup \{h(\mu)- h(\nu): \pi\mu =\nu\} \\
\notag  &= h(\nu ) +\int_Y P(\pi ,0)\,  d \nu \ .
\end{align}

In \cite{PetersenShin2005} a finite-range, combinatorial approach
was developed for the relative pressure and entropy, in which
instead of examining entire infinite sequences $x$ in each fiber
over a given point $y \in \Omega_B$, it is enough to deal just with
preimages of finite blocks (which may or may not be extendable to
full sequences in the fiber). For each $n = 1, 2, \dots$ and $y \in
Y$ let $E_n (y)$ be a set consisting of exactly one point from each
nonempty cylinder $x[0,n-1] \subset \pi^{-1} y[0,n-1]$.
\medskip
Then for each $V \in C(\Omega_A)$,
 \be P(\pi,V) (y) = \limsup_{n
 \rightarrow \infty} \frac{1}{n} \log \bigg[ \sum_{x \in E_n (y)}
 \exp \Big( \sum_{i=0}^{n-1} V (\sigma^i x) \Big) \bigg]
 \en
 {\em a.e. with respect to every
ergodic invariant measure on Y}. Thus, we obtain the
value of $P(\pi,V) (y)$ a.e. with respect to every
ergodic invariant measure
on $Y$ if we {delete from the definition of $D_n (y)$ the
requirement that $x \in \pi^{-1} (y)$}.

In particular, the relative topological entropy is given by
 \be P(\pi ,0)(y)=\limsup_{n \to \infty}\frac{1}{n}\log
 |\pi^{-1}y[0,n-1]|
 \en
 a.e. with respect to every ergodic
invariant measure on $Y$.

And if $\mu$ is relatively maximal over $\nu$, in the sense that it
achieves the supremum in (\ref{eq_maxent}), then
 \be h_\mu(X|Y)=\int_Y \lim_{n \to \infty}\frac{1}{n}
 \log |\pi^{-1}y[0,n-1]|\, d\nu (y) .
 \en

\subsection{Measures of maximal and relatively maximal entropy}\label{sec_relmaxent}

Already Shannon \cite{ShannonWeaver1949} constructed the measures of
maximal entropy on irreducible shifts of finite type. Parry
\cite{Parry1964} independently and from the dynamical viewpoint
rediscovered the construction and proved uniqueness. For an
irreducible shift of finite type the unique measure of maximal
entropy is a 1-step Markov measure whose transition probability
matrix is the stochasticization, as in
(\ref{stochasticization}),
of the $0,1$ matrix that defines the subshift. When studying factor
maps $\pi : \Omega_A \to \Omega_B$ it is natural to look for {\em
measures of maximal relative entropy}, which we also call {\em
relatively maximal measures }: for fixed $\nu$ on $\Omega_B$, look
for the $\mu\in\pi^{-1}\nu$ which have {maximal entropy in that
fiber}. Such measures always exist by compactness and upper
semicontinuity, but, in contrast to the Shannon-Parry case (when
$\Omega_B$ consists of a single point), they need not be unique.
E.g., in Example \ref{ex_Walters}, the two-to-one map $\pi$
respects entropy, and for $p\neq 1/2$  there
are exactly two ergodic measures
(the Bernoulli measures $\mu_p$ and $\mu_{1-p}$) which
$\pi$ sends to $\nu_p$. Moreover, there exists some
 $V_p \in \cc(Y)$ which has $\nu_p$ as a unique equilibrium
state \cite{Israel1979, Phelps2002}, and  $V_p\circ \pi$ has exactly
two {ergodic} equilibrium states,
 $\mu_p$ and $\mu_{1-p}$.

Here is a useful characterization of relatively maximal measures due
to Shin.

\begin{thm}[\cite{Shin2001}]
 Suppose that $\pi :X \to Y$ is a factor map of shifts of finite type,
 $\nu \in \cm(Y)$ is ergodic, and $\pi \mu = \nu$.
Then $\mu$ is relatively maximal over $\nu$ if and only if there is
$V \in \cc (Y,\Br)$ such that $\mu$ is an equilibrium state of $V
\circ \pi$.
\end{thm}

If there is a {\em locally constant} saturated compensation function
$G \circ \pi$, then every {Markov} measure on $Y$ has a unique
relatively maximal lift, which is Markov, because then the
relatively maximal measures over an equilibrium state of $V \in \cc
(Y,\Br)$ are the equilibrium states of $V \circ \pi + G \circ \pi$
\cite{Walters1986}.
Further, the measure of maximal entropy $\max_X$ is the unique
equilibrium state of the potential function 0 on $X$; and the
relatively maximal measures over $\max_Y$ are the equilibrium states
of $G \circ \pi$.

It was proved in \cite{PetersenQuasShin2003} that for each ergodic
$\nu$ on $Y$, there are only a finite number of relatively maximal
measures over $\nu$.
 In fact, for a 1-block factor map $\pi$ between 1-step
shifts of finite type $X,Y$,
the number of ergodic
  invariant measures of maximal entropy in the fiber
  $\pi^{-1}\{\nu\}$ is at most
\be N_\nu(\pi)= \min\{|\pi^{-1}\{b\}|: b \in \ca (Y), \nu [b]>0\}.
\en

This follows from the theorem in \cite{PetersenQuasShin2003} that
for each ergodic $\nu$ on $Y$, any two distinct ergodic measures on
$X$ of maximal entropy in the fiber $\pi^{-1}\{\nu\}$ are {{\em
relatively orthogonal}}. This concept is defined as follows.

For $\mu_1 , \dots ,\mu_n \in \cm (X)$ with $\pi \mu_i= \nu$ for all
$i$, their {\em relatively
  independent joining} $\hat\mu$ over $\nu$ is defined by:

if $A_1, \dots , A_n$ are measurable subsets of $X$ and $\cf$ is the
$\sigma$-algebra of $Y$, then
 \be
    \label{eq:relind}
    \hat\mu(A_1\times\ldots\times A_n)=
    \int_Y \prod_{i=1}^n
    \mathbb E_{\mu_i}(\one _{A_i}|\pi^{-1}\mathcal F)\circ\pi^{-1}
    \,d\nu
  \en
in which $\mathbb E$ denotes conditional expectation.
Two ergodic measures $\mu_1 , \mu_2$ with $\pi \mu_1 = \pi \mu_2 =
\nu$ are {\em relatively orthogonal} (over $\nu$),
  $\mu_1 \perp_\nu \mu_2$, if
\be (\mu_1 \otimes_\nu  \mu_2) \{ (u,v) \in X \times X : u_0
=v_0\}=0. \en This means that with respect to the relatively
independent joining or coupling, there is {zero probability of
coincidence of symbols in the two coordinates}.

That the second theorem (distinct ergodic relatively maximal
measures in the same fiber are relatively orthogonal) implies the
first (no more than $N_{\nu}(\pi)$ relatively maximal measures over
$\nu$)
follows from the Pigeonhole Principle. If we have $n
> N_\nu(\pi)$ ergodic measures $\mu_1,\ldots,\mu_n$
  on $X$, each projecting to $\nu$ and each of maximal entropy in the
  fiber $\pi^{-1}\{ \nu\}$,
we form the relatively independent joining $\hat\mu$ on $X^n$  of
the
  measures $\mu_i$ as above.
Write $p_i$ for the projection $X^n\to X$ onto the $i$'th
  coordinate.
 For $\hat\mu$-almost every $\hat x$ in
  $X^n$, $\pi(p_i(\hat x))$ is independent of $i$;
abusing notation for simplicity, denote it by $\pi
  (\hat x)$.
  Let $b$ be a symbol in the alphabet of $Y$ such that $b$ has
  $N_\nu(\pi)$ preimages $a_1,\dots ,a_{N_\nu(\pi)}$ under the block map $\pi$.
Since $n > N_\nu(\pi )$, for every $\hat x \in \pi^{-1}[b]$ there
are $i \neq j$ with $(p_i \hat x)_0 = (p_j \hat x)_0$. At least one
of the sets $S_{i,j} = \{ \hat x \in X^n : (p_i \hat x)_0 = (p_j
\hat x)_0 \}$ must have positive $\hat\mu$-measure, and then also
\be (\mu_i \otimes _\nu \mu_{j}) \{(u,v) \in X \times X: \pi u = \pi
v, u_0 = v_0 \} > 0, \en
 contradicting relative orthogonality.
(Briefly, if you have more measures than preimage symbols, two of
those measures have to coincide on one of the symbols: with respect
to each measure, that symbol a.s. appears infinitely many times in
the same place.)

The second theorem is proved by ``interleaving" measures to increase
entropy.
  If there are two relatively maximal measures over $\nu$ which are
  not relatively orthogonal, then the measures can be
  `mixed' to give a measure with greater entropy.
We concatenate words from the two processes, using the fact that the
two measures are supported on sequences that agree infinitely often.
Since $X$ is a 1-step SFT, we can switch over whenever a
{coincidence} occurs.
That the switching {increases entropy} is seen by using the strict
concavity of the function $-t \log t$ and lots of calculations with
conditional expectations.

\begin{ex}\label{ex_nosofics}

Here is an example (also discussed in \cite[Example
1]{PetersenQuasShin2003}) showing that to find relatively maximal
measures over a Markov measure it is not enough to consider only
sofic measures which map to it. We describe a factor map $\pi$ which
is both left and right e-resolving (see section \ref{sec_resolving})
and such that there is a unique relatively maximal measure $\mu$
above any fully-supported Markov measure $\nu$, but the measure
$\mu$
is not Markov, and it is not even sofic.

We use vertex shifts of finite type.
    The alphabet for the domain subshift is $\{a_1, a_2, b\}$
    (in that order for indexing purposes),
    and the factor map (onto the 2-shift $(\Omega_2,\sigma)$) is the
    1-block code $\pi$ which
    erases subscripts. The transition diagram and matrix $A$ for the domain
    shift of finite type $(\Omega_A,\sigma)$ are
    \be
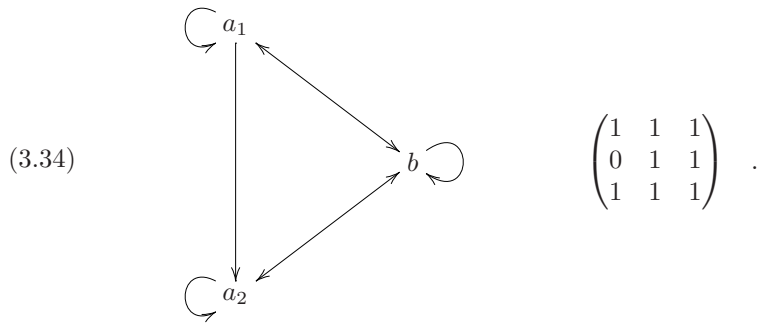
    \begin{gathered}
    \mbox{
    \xymatrix{
    a_1\ar@(ul,dl)[]\ar@{<->}[drr]\ar@{->}[dd]\\
    &&b\ar@(ur,dr)[]
  &&\mbox
    {
    $\begin{pmatrix}
    1 & 1 & 1\\
    0 & 1 & 1\\
    1 & 1 &1
    \end{pmatrix}$\quad .}
    \\
    a_2\ar@(ul,dl)[]\ar@{<->}[urr]
    }}
    \end{gathered}
    \en
Above the word $ba^nb$ in $\Omega_2$ there are $n+1$ words in
$\Omega_A$: above $a^n$ we see $k$ $a_1$'s followed by $n-k$
$a_2$'s, where $0 \leq k \leq n$. Let us for simplicity consider the
maximal measure $\nu$ on $(\Omega_2,T)$; so, $\nu (\mathcal
C_0(ba^nb))
=2^{-n-2}$. Now the maximal entropy lift $\mu$ of $\nu$ will assign
equal measure $2^{-(n+2)}/(n+1)$ to each of the preimage blocks
of $ba^nb$.
If $\mu$ is sofic, then
(as in Sec. \ref{linrepser})
there are vectors $u,v$ and a square  matrix
$Q$ such that
$\mu (\mathcal C_0(b(a_1)^nb)= uQ^nv$ for all $n>0$.
Then
the function $n\mapsto uQ^nv$ is some finite sum of terms
of the form $r n^j (\lambda^n)$ where $j\in \mathbb Z_+$ and $r,
\lambda$ are constants. The function
$n\mapsto 2^{-(n+2)}/(n+1)$
 is
not a function of this type.
\end{ex}

\begin{prob}\label{prob_liftmarkovs}
Is it true that for every factor map $\pi : \Omega_A \to \Omega_B$
every (fully supported) Markov measure $\nu$ on $\Omega_B$ has a
unique relatively maximal measure that maps to it, and this is also
a measure with full support?
\end{prob}

\begin{rem} After the original version of this paper was posted on the
Math Arxiv and submitted for review, we received the preprint
\cite{Yoo2009} of Jisang Yoo containing the following result: "Given
a factor map from an irreducible SFT X to a sofic shift Y and an
invariant measure $\nu$ on Y with full support, every measure on X
of maximal relative entropy over $\nu$ is fully supported." This
solves half of Problem \ref{prob_liftmarkovs}.
\end{rem}

\subsection{Finite-to-one codes}\label{sec_finitetoone}

Suppose $\pi : \Omega_A \to \Omega_B$ is a finite-to-one factor map
of irreducible shifts of finite type. There are some special
features of this case which we collect here for mention. Without
loss of generality, after recoding we assume that $\pi$ is a 1-block
code. Given a Markov measure $\mu$ and a periodic point $x$ we
define the {\it weight-per-symbol} of $x$ (with respect to $\mu$) to
be \be \label{wps} \textnormal{wps}_{\mu}(x) :=\lim_{n \to
\infty}\frac 1n \log \mu \{y: x_i = y_i, 0 \leq i < n\} \ . \en

\begin{prop} \label{prop_finitetoone}
Suppose
 $\pi : \Omega_A \to \Omega_B$ is a finite-to-one
factor map of irreducible shifts of finite type. Then

\begin{enumerate}
\item
The measure of maximal entropy on $\Omega_B$
lifts to the measure of maximal entropy on $\Omega_A$.
\item
Every Markov measure on $\Omega_B$ lifts to a {unique}
 Markov measure of equal order on
$\Omega_A$.
\item  \label{percond}
If $\mu, \nu$  are Markov measures
 on $\Omega_A,\Omega_B$ respectively, then
 the following are equivalent:
\begin{enumerate}
\item
  $\pi \mu = \nu$
\item
for every periodic point $x$ in $\Omega_A$,
$\textnormal{wps}_{\mu}(x) =
\textnormal{wps}_{\nu}(\pi x) $.
\end{enumerate}
\end{enumerate}
\end{prop}

Proofs can be found in, for example, \cite{Kitchens1982}. For
infinite-to-one codes, we do not know an analogue of Prop.
\ref{prop_finitetoone} (\ref{percond}).

\subsection{The semigroup measures of Kitchens and Tuncel}\label{semigroup}

{There is}
 a hierarchy of sofic
measures according to their sofic degree. Among the degree-1 sofic
measures, there is a distinguished and very well behaved subclass,
properly containing the Markov measures. These are the {\em
semigroup measures} introduced and studied by Kitchens and Tuncel in
their memoir \cite{KitchensTuncel1985}.
Roughly speaking,  semigroup measures are to Markov measures as
sofic subshifts are to SFT's.

A sofic subshift can be presented by a semigroup \cite{Weiss1973,
KitchensTuncel1985}. Associated to this are nonnegative transition
matrices $R_0,L_0$. A semigroup measure (for the semigroup
presentation) is defined by
 a state probability vector and
a pair of stochastic matrices $R,L$ with 0/+ pattern matching
$R_0,L_0$ and satisfying certain consistency conditions.
 These matrices can be multiplied to compute measures of cylinders. A
measure is a semigroup measure if there exist a semigroup and
apparatus as above which can present it.
 We will not review this constructive part of the theory, but
 we mention some alternate characterizations of these
measures.

For a sofic measure $\mu$ on $X$ and a periodic point $x$
in $X$, the weight-per-symbol of $x$ with respect to
$\mu$ is still well defined by (\ref{wps}).
Let us say a factor map
$\pi$ {\em respects $\mu$-weights}
 if whenever $x,y$ are periodic points with the same image
we have $\textnormal{wps}_{\mu}(x)=\textnormal{wps}_{\mu}(y)$. Given
a word $U=U[-n\dots 0]$ and a measure $\mu$, let $\mu_U$ denote the
conditional measure on the future, i.e. if $UW$ is an allowed word
then $\mu_U (W)= \mu( UW)/\mu (U)$.

\begin{thm}\cite{KitchensTuncel1985}
Let $\nu$ be a shift-invariant measure on an irreducible sofic
subshift $Y$. Then the following are equivalent:
\begin{enumerate}
\item
$\nu$ is a semigroup measure.
\item
$\nu$ is the image of a Markov measure $\mu$ under a finite-to-one
factor map which respects $\mu$-weights.
\item
$\nu$ is the image of a Markov measure $\mu$
under a degree 1 resolving factor map which respects
$\mu$-weights.
\item
The collection of conditional measures $\mu_U$, as
$U$ ranges over all $Y$-words, is finite.
\end{enumerate}
\end{thm}

There is also a thermodynamic characterization of these measures as
unique equilibrium states of  bounded Borel functions which are
locally constant on doubly transitive points, very analogous to the
characterization of Markov measures as unique equilibrium states of
continuous locally constant functions.
 The semigroup measures satisfy other nice properties as well.

\begin{thm}\cite{KitchensTuncel1985}
Suppose $\pi:X\to Y$ is a finite-to-one factor map of irreducible
sofic subshifts and $\mu$ and  $\nu$ are semigroup measures on $X$
and $Y$ respectively. Then
\begin{enumerate}
\item
$\nu$ lifts by $\pi$ to a unique semigroup measure on $X$, and this
is the unique ergodic measure on $X$  which maps to $\nu$;
\item
$\pi\mu$ is a semigroup measure if and only if $\pi$ respects
$\mu$-weights;
\item
there is an irreducible sofic subshift $X'$ of $X$ such that $\pi$
maps $X'$ finite-to-one onto $X$ \cite{MarcusPetersenWilliams1984},
and therefore $\nu$ lifts to a semigroup measure  on $X'$.
\end{enumerate}
\end{thm}
In contrast to the last statement, it can happen for {an
infinite-to-one} factor map between irreducible SFTs that there is a
Markov measure on the range which cannot lift to a Markov measure on
any subshift of the domain \cite{MarcusPetersenWilliams1984}.

We finish here with
an example. There are others in
\cite{KitchensTuncel1985}.

\begin{ex}
This is an example of a finite-to-one, one-to-one a.e. 1-block code
$\pi : \Omega_A \to \Omega_B$ between mixing vertex shifts of finite
type, with a 1-step Markov measure $\mu$ on $\Omega_A$,  such that
the following hold:
\begin{enumerate}
\item
For all periodic points $x,y$ in $\Omega_A$, $ \pi x = \pi y \text{
implies that } \textnormal{wps}_{\mu}(x) = \textnormal{wps}_{\mu}(y)
\ . $
\item
$\pi \mu $ is not Markov on $\Omega_B$.
\end{enumerate}
Here the alphabet of $\Omega_A$ is $\{1,2,3\}$; the alphabet of
$\Omega_B$ is $\{1,2\}$;
 \begin{equation*}
 A= \begin{pmatrix}
    0 & 1 & 0\\
    1 & 0 & 1\\
    1 & 1 & 0
    \end{pmatrix}
\qquad\text{ and }\qquad B=
\begin{pmatrix}
    0 & 1 \\
    1 & 1
    \end{pmatrix};
    \end{equation*}
and $\pi$ is the 1-block code sending $1$ to $1$ and sending $2$ and
$3$ to $2$. The map $\pi$ collapses the points in the orbit of
$(23)^*$ to a fixed point and collapses no other periodic points.
(Given a block $B$, we let $B^*$ denote a periodic point obtained by
infinite concatenation of the block $B$.)

Let $f$ be the function on
 $\Omega_A$ such that $f(x) = \log 2$ if $x_0x_1=23$,
$f(x) = \log (1/2)$ if $x_0x_1=32$ and $f(x) = 0$ otherwise. Let $\mu$ be
the 1-step Markov measure which is the unique equilibrium state for
$f$, defined by the stochasticization $P$ of the matrix
\begin{equation*}
M =
 \begin{pmatrix}
    0 & 1 & 0\\
    1 & 0 & 2\\
    1 & 1/2 & 0
    \end{pmatrix}.
    \end{equation*}
Let $\lambda$ denote the spectral radius
of $M$.
Suppose that $\nu = \pi \mu$ is Markov, of any order.
Then $
\textnormal{wps}_{\nu}(2^*) =
\textnormal{wps}_{\mu}((23)^*) =
-\log \lambda $. Also, there must be a constant
$c$ such that for all large $n$,
\be
\textnormal{wps}_{\nu}((12^n)^*) =
\frac 1{n+1}(c + (n+1)\textnormal{wps}_{\nu}(2^*)) =
\frac {c}{n+1} -\log \lambda
\ .
\en
So,  for all large $n$,
\be
\begin{aligned}
\frac {c}{2n+1} - \log \lambda
&= \textnormal{wps}_{\nu}((12^{2n})^*) =
\textnormal{wps}_{\mu}((1(23)^n)^*)
= \frac 1{2n+1}\log (2\lambda^{-(2n+1)})
 \\
 &\text{and}\\
\frac {c}{2n+2}-\log \lambda
&= \textnormal{wps}_{\nu}((12^{2n+1})^*) =
\textnormal{wps}_{\mu}((1(23)^n2)^*) =
\frac {1}{2n+2}\log (\lambda^{-(2n+2)})\ .
\end{aligned}
\en
Thus $c=\log 2 $ and $c=0$, a contradiction.
Therefore $\pi \mu $ is not Markov.
\end{ex}

\section{Identification of hidden Markov measures}\label{sec_ident}

Given a finite-state stationary process, how can we tell whether it
is a hidden Markov process? If it is, how can we construct some
Markov process of which it is a factor by means of a sliding block
code? When is the image of a Markov measure under a factor map again
a Markov measure? These questions are of practical importance, since
scientific measurements often capture only partial information about
systems under study, and in order to construct useful models the
significant hidden variables must be identified and included.
Beginning in the 1960's some criteria were developed for recognizing
a hidden Markov process: loosely speaking, an abstract algebraic
object constructed from knowing the measures of cylinder sets should
be in some sense finitely generated. Theorem \ref{th_rationalequiv}
below gives equivalent conditions, in terms of formal languages and
series (the series is ``rational"), linear algebra (the measure is
``linearly representable"), and abstract algebra (some module is
finitely generated), that a shift-invariant probability measure be
the image under a 1-block map of a shift-invariant 1-step Markov
measure. In the following we briefly explain this result, including
the terminology involved.

Kleene \cite{Kleene1956} characterized rational languages as the
linearly representable ones, and this was generalized to formal
series by Sch\"{u}tzenberger \cite{Schutzenberger1961}. In the study
of stochastic processes, functions of Markov chains were analyzed by
Gilbert \cite{Gilbert1959}, Furstenberg \cite{Furstenberg1960},
Dharmadhikari
\cite{Dharma1963,Dharma1963-2,Dharma1964,Dharma1965,Dharma1968,
DharmaNadkarni1970},
Heller \cite{Heller1965, Heller1967}, and others. For the connection
between rational series and continuous images of Markov chains, we
follow Berstel-Reutenauer \cite{BerstelReutenauer1988} and
Hansel-Perrin \cite{HanselPerrin1989}, with an addition to explain
how to handle zero entries. Subsequent sections describe the
approaches of Furstenberg and Heller and related work.

 Various problems around these ideas were (and
continue to be) explored and solved. In particular, it is natural to
ask when is the image of a Markov measure $\mu$ under a continuous
factor map $\pi$ a Gibbs measure (see (\ref{Eq_BowenGibbs}), or when
 is the image of a Gibbs measure again a Gibbs measure?
 Chazottes and Ugalde \cite{ChazottesUgalde2003} showed that
 if $\mu$ is $k$-step Markov on a full shift $\Omega_d$ and
  $\pi$ maps $\Omega_d$ onto another full shift $\Omega_D$,
  then the image $\pi\mu$ is a
Gibbs measure which is the unique equilibrium state of a H\"older
continuous potential which can be explicitly described in terms of a
limit of matrix products and computed at periodic points. They also
gave sufficient conditions in the more general case when the factor
map is between SFT's. The case when $\mu$ is Gibbs but not
necessarily Markov is considered in \cite{ChazottesUgalde2009}. For
higher-dimensional versions see for example \cite{KunschGeman1995,
MaesVandevelde1995, Haggstrom2003}.

Among the extensive literature that we do not cite elsewhere, we can
mention in addition \cite{Harris1955,Nasu1985,Fannes1992,
BezhaevaOseledets2005,Schonhuth2009}.

\subsection{Formal series  and formal languages}\label{sec_lang}

\subsubsection{Basic definitions}

As in Section \ref{sec_sfts}, continue to let $\ca$ be a finite
alphabet, $\ca^{*}$ the set of all finite words on $\ca$, and
$\ca^{+}$ the set of all finite nonempty words on $\ca$. Let
$\epsilon$ denote the empty word.  A \em language \em~ on $\ca$ is
any subset $\mathcal{L}\subset \ca^{*}$.

Recall that a \em monoid \em~ is a set $S$ with a binary operation
$S \times S \rightarrow S$ which is associative and has a neutral
element (identity).  This means we can think of $\ca^{*}$ as the
multiplicative free monoid generated by $\ca$, where the operation
is concatenation and the neutral element is $\epsilon$.

A \em formal series \em~ (nonnegative real-valued, based on $\ca$)
is a function $s: \ca^{*} \rightarrow {\R}_{+}$. For all $ w \in
\ca^{*}, \: s(w)=(s, w) \in {\R}_{+}$, which can be thought of as
the coefficient of $w$ in the series $s$.  We will think of this $s$
as $\sum_{w\in \ca^{*}}s(w)w$, and this will be justified later. If
$v\in \ca^{*}$ and $s$ is the series such that $s(v)=1$ and $s(w)=0$
otherwise, then we sometimes use simply
 $v$ to denote  $s$.

Associated with any language $\cl$ on $\ca$ is its {\em
characteristic series} $F_\cl :\ca^* \to \Br_+$ which assigns 1 to
each word in $\cl$ and 0 to each word in $\ca^* \setminus \cl$.
Associated to any Borel measure $\mu$ on $\ca^{\Bz_+}$ is its {\em
corresponding series} $F_\mu$ defined by
 \be
 F_\mu(w) = \mu(\cc_0(w)) = \mu\{x \in \ca^{\Bz_+}: x[0,|w|-1]=w\} .
 \en

It is sometimes useful to consider formal series with values in any
{\em semiring} $K$, which is just a ring without subtraction.  That
is, $K$ is a set with operations $+$ and $\cdot$ such that $(K, +)$
is a commutative monoid with identity element $0$, $(K,\cdot)$ is a
monoid with identity element $1$; the product distributes over the
sum; and for $k \in K$, $0k=k0=0$.

We denote the set of all $K$-valued formal series based on $\ca$ by
$K \langle \langle {\ca} \rangle \rangle$ or $\mathcal{F}_{K}(\ca)$.
We further abbreviate ${\R}_{+} \langle \langle {\ca} \rangle
\rangle = \mathcal {F}(\ca) $.

Then $\mathcal {F}(\ca)$ is a semiring in a natural way: For $f_{1},
f_{2} \in \mathcal{F} (\ca)$, define
\begin{enumerate}
\item $(f_{1} + f_{2}) (w) = f_{1}(w)+f_{2}(w)$
\item $(f_{1}f_{2})(w) = \sum f_{1}(u)f_{2}(v)$,
where the sum is over all $u, v \in \ca^{*}$ such that $uv=w$, a
finite sum.
 \end{enumerate}
The neutral element for multiplication in $\F$ is
 \be
 s_{1}(w) = \begin{cases}1 & \text{if $w=\epsilon$} \\0 &
 \text{otherwise.}
 \end{cases}
 \en
As discussed above, we will usually write simply
$\epsilon $ for $s_1$.
There is a natural injection ${\R}_{+} \hookrightarrow \F$
defined by $t
\mapsto t \epsilon$ for all $t\in {\R}_{+}$.

Note that:
\begin{itemize}
\item ${\R}_{+}$ acts on $\F$ on both sides: \\
      $(ts)(w) = ts(w), \; (st)(w) = s(w)t, \; $ for all
      $w \in \ca^{*}$, for all $t \in {\R}_{+}$.
\item There is a natural injection $\ca^{*} \hookrightarrow \F$ as a
multiplicative submonoid: For $w\in \ca^{*}$ and $v \in \ca^{*}$,
define
    $$w(v)= \delta_{wv} = \begin{cases}      1 & \text{if $w=v$} \\
   0 & \text{otherwise.}       \end{cases}$$
This is a 1-term series.
\end{itemize}
\begin{defn}
The \em{support} \em~ of a formal series $s \in \F$ is
$$\supp(s) = \{ w \in \ca^{*}: s(w)\neq 0\} .$$
\end{defn}
Note that $\supp(s)$ is a language. A language corresponds to a
series with coefficients 0 and 1, namely its characteristic series.
\begin{defn}\label{defn_poly}
A \em polynomial \em~ is an element of $\F$ whose support is a
finite subset of $\ca^{*}$. Denote the $K$-valued polynomials based
on $\ca$ by $\wp_{K}(\ca)=K \langle \ca \rangle$. The \em degree
\em~ of a polynomial $p$ is $\deg(p) = \max \{|w|:p(w) \neq 0\}$ and
is $- \infty$ if $p \equiv 0$.
\end{defn}
\begin{defn}  A family $\{f_{\lambda}: \lambda \in \Lambda\} \subset \F$ of series
is called \em locally finite \em~ if for all $w \in \ca^{*}$ there
are only finitely many $\lambda \in \Lambda$ for which
$f_{\lambda}(w)\neq 0$. A series $f \in \F$ is called \em proper
\em~ if $f(\epsilon)=0$.
\end{defn}
\begin{prop}
If $f \in \F$ is proper, then $\{f^{n}: n=0, 1, 2, \dots \}$ is
locally finite.
\end{prop}
\begin{proof}  If $n>|w|$, then $f^{n}(w) = 0$, because \\
$$f^{n}(w) = \displaystyle \sum_{\substack{
u_{1} \dots u_{n} = w\\
 u_i \in \ca^{*}, \; i=1, \dots , n}
 }f(u_{1}) \dots f(u_{n}
)$$ and at least one $u_{i}$ is $\epsilon$.
\end{proof}
\begin{defn}
If $f \in \F$ is proper, define $$f^{*} = \sum_{n=0}^\infty f^{n}
\textrm{ and } f^{+} = \sum_{n=1}^\infty f^{n} \; \textrm{(a
pointwise finite sum)},$$ with $f^0 = 1 = 1 \cdot \epsilon =
\epsilon$.
\end{defn}

\subsubsection{Rational series and languages}

\begin{defn}
The \em rational operations \em~ in $\F$ are sum $(+)$, product
$(\cdot)$, multiplication by real numbers $(tw)$, and $*: f
\rightarrow f^{*}$. The family of \em rational series \em~ consists
of those $f \in \F$ that can be obtained by starting with a finite
set of polynomials in $\F$ and applying a finite number of rational
operations.
\end{defn}
\begin{defn}
A language $\mathcal{L} \subset \ca^{*}$ is \em rational \em~ if and
only if its characteristic series
 \be
 F(w) =    \begin{cases}   1 & \text{if $w \in \mathcal{L}$} \\
     0 &    \text{if $w \notin \mathcal{L}$}
     \end{cases}
     \en
is rational.
\end{defn}

Recall that regular languages correspond to {\em regular
expressions}: The set of regular expressions includes $\ca, \;
\epsilon, \; \emptyset$ and is closed under $+, \; \cdot$ , *. A
language recognizable by a finite-state automaton, or consisting of
words obtained by reading off sequences of edge labels on a finite
labeled directed graph, is regular.

\begin{prop}
A language $\mathcal{L}$ is rational if and only if it is \em
regular. \em~ Thus a nonempty
 insertive and extractive language is rational
if and only if it is the language of a sofic subshift.
\end{prop}

\subsubsection{Distance and topology in $\F$}
If $f_{1}, f_{2} \in \F,$ define \be D(f_{1}, f_{2}) = \inf \{ n
\geq 0: \; \text{there} \, \text{is} \; w \in \ca^{n} \textrm{ such
that } f_{1}(w) \neq f_{2}(w) \} \en and \be d(f_{1}, f_{2}) = \frac
{1} {2^{D(f_{1}, f_{2})}}. \en Note that $d(f_{1}, f_{2})$ defines
an \em ultrametric \em~ on $\F$:
 \be d(f, h) \leq \max \{ d(f,g),d(g, h) \} \leq d(f,g) +d(g,h) .
 \en
 With respect to the metric
 $d$, $f_{k} \rightarrow f$ if and only if
 for each $w \in \ca^*$, $f_{k}(w) \rightarrow
f(w)$ in the discrete topology on $\R$, i.e.
 $f_k(w)$ eventually equals $f(w)$.
\begin{prop}
$\F$ is \em complete \em~ with respect to the metric $d$ and is a
\em topological semiring \em~ with respect to the metric $d$ (that
is, $+$ and $\cdot$ are continuous as functions of two variables).
\end{prop}
\begin{defn}
A family $\{ F_{\lambda}: \lambda \in \Lambda \} $ of formal series
is called \em summable \em~ if there is a series $F \in \F$ such
that for every $\delta > 0$ there is a finite set $\Lambda_{\delta}
\subset \Lambda$ such that for each finite set $I \subset \Lambda$
with $\Lambda_{\delta} \subset I$, $d(\sum_{i \in I}F_i, F)<
\delta$. Then $F$ is called the \em sum \em~ of the series and we
write
$F= \sum_{\lambda \in \Lambda} F_{\lambda}$.\\
\end{defn}
\begin{prop}
If $\{F_\lambda: \lambda \in \Lambda\}$ is locally finite, then it
is summable, and conversely.
\end{prop}
Thus any $F \in \F$ can be written as $F=\sum_{w \in \A}F(w)w$,
where the formal series is a convergent infinite series of
polynomials in the metric of $\F$. Recall that
$$(F(w)w) (v) =  \begin{cases}   F(w) & \text{if $w=v$ }\\   0 &
\text{if $w \neq v,$}    \end{cases}$$ where $F(w)w \in \F$ and
$w\in \A$, so that  $\{F(w)w: w \in \A\}$ is a
locally finite, and hence summable, subfamily of $\F$.\\

We note here that the set $\wp (\ca)$ of all polynomials is dense in $\F$.\\

\subsubsection{Recognizable (linearly representable)
series}\label{linrepser}

\begin{defn}\label{def_linrep}
$F \in \F$ is \em linearly representable \em~ if there exists an $n
\geq 1$ (the \em dimension \em~ of the representation) such that
there are a $1 \times n$ nonnegative row vector $x \in \R_{+}^{n}$,
an $n \times 1$ nonnegative column vector $y \in {\R}_{+}^{n}$, and
a morphism of multiplicative monoids $\phi: \A \rightarrow \R_{+}^{n
\times n}$
 (the multiplicative monoid of nonnegative $n \times n$ matrices) such that for all
 $w \in \A$, $F(w) = x \phi(w)
 y$ (matrix multiplication).
 A {\em linearly representable measure} is one whose associated
 series is linearly representable.
 The triple $(x,\phi,y)$ is called the {\em linear representation} of
 the series (or measure).
\end{defn}

\begin{ex} Consider a Bernoulli measure $\mathcal{B}(p_{0}, p_{1}, \dots , p_{d-1})$ on
 $\Omega_{+}(\ca) = \ca^{\Z_{+}}$ where $\ca =  \{a_{0}, a_{1},
 \dots, a_{d-1}\}$, and
  $p =(p_{0}, p_{1}, \dots , p_{d-1})$ is a probability vector.  Let
  $f = \sum_{i=0}^{d-1}p_{i}a_{i} \in \F$.
 Then
 $$f(w) =  \begin{cases}    p_{i}& \text{if $w = a_{i}$ }\\ 0 &
 \text{if $w \neq a_{i}$}  .  \end{cases} $$
 Define $F_{p} = f^{*} = \sum_{n \geq 0} f^{n}$. Note that $f$ is proper
 since we have $f(\epsilon) = 0$.
 Consider the particular word
 $w = a_{2}a_{0}$. Then
  $f^{0}(w) = f(w) =0,$ and
 for $n \geq 3$, we have $f^{n}(w) = 0$ because any factorization $w=u_{1}u_{2}u_{3}$
 includes
 $\epsilon$ and $f(\epsilon) = 0$.
 Thus $F_p(w)=f^*(w)=f^{2}(w) = \sum_{uv =w}f(u)f(v) = f(a_{2})f(a_0) = p_{2}p_{0}$.
 Continuing in this way, we see that for $w_{i} \in \ca$,
 $F_{p}(w_{1} w_{2} \dots w_{n}) = p_{w_{1}}p_{w_{2}} \dots p_{w_{n}}$.
\end{ex}

{
\begin{ex}\label{ex_linrepofmarkov}
 Consider a Markov measure $\mu$ on $\Omega_{+}(\ca)$ defined by a $d
\times d$ stochastic matrix $P$ and a $d$-dimensional probability
row vector $p=(p_0, p_1, \cdots, p_{d-1})$. Define $F_{p,P} \in
\cf(\ca)$ by $F_{p,P}(w_1\dots w_n)=\mu(\cc_0(w_1 \dots w_n))$ for
all $w_1, \dots ,w_n \in \ca$.
 Put
$y=(1, \dots ,1)^\text{tr} \in \R_{+}^{d}, \; x=p \in \R_{+}^{d},$
and let $\phi$ be generated by $\phi(a_{j})$, $j=0,1,...,d-1$, where
\be \phi(a_{j})= \left(
\begin{matrix}
0   &   \cdots  & {P}_{0 j} &     0 & \cdots & 0\\
0   &   \cdots  &{P}_{1j} &       0 & \cdots & 0\\
\vdots  &   \cdots  &\vdots &     \vdots   & \cdots & \vdots\\
0   & \cdots        & {P}_{d-1,j}&    0   & \cdots & 0
\end{matrix}
\right) \; \text{for each} \; {a}_j \in \ca. \en Then the triple
$(x, \phi , y)$ represents the given Markov measure $\mu$.
  In this Markov case each matrix $\phi(a_j)$ has at most one
  nonzero column and thus has rank at most 1.
\end{ex}
}

{
\begin{ex}\label{ex_linrepofsofic}
Now we show how to obtain a linear representation of a sofic measure
that is the image under a 1-block map $\pi$ of a 1-step Markov
measure. Let $\mu$ be a 1-step Markov measure determined by a $d
\times d$ stochastic matrix $P$ and fixed vector $p$ as in Example
\ref{ex_linrepofmarkov}. Let $\pi :X \to Y$ be a 1-block map from
the SFT $X$ to a subshift $Y$. For each $a$ in the alphabet
$B=\ca(Y)$ let $P_a$ be the $d \times d$ matrix such that \be
P_a(i',j')=
\begin{cases}
P(i',j') &\quad\text{if } \pi(j')=a\\
0 &\quad\text{otherwise.}
\end{cases}
 \en
 Thus $P_a$ just zeroes out all the columns of $P$ except the ones
 corresponding to indices in the $\pi$-preimage of the symbol $a$ in the
 alphabet of $Y$.
 Again let $y=(1, \dots ,1)^\text{tr}$.
For each $a \in B$ define $\phi(a)=P_a$. That the $\nu$-measure of
each cylinder in $Y$ is the sum of the $\mu$-measures of its
preimages under $\pi$ says that
  the triple $(x,\phi,y)$
 represents $\nu=\pi\mu$.
   \end{ex}
}

 In working with linearly representable measures, it is useful to
know that the nature of the vectors and matrix involved in the
representation can be assumed to have a particular restricted form.
{}{Below, we say a matrix $P$ is a direct sum of irreducible
stochastic matrices if the index set for the rows and columns of $P$
is the disjoint union of sets for which the associated principal
submatrices of $P$ are irreducible stochastic matrices.
(Equivalently, there are irreducible stochastic matrices $P_1, \dots
, P_k$ and a permutation matrix $Q$ such that $QPQ^{-1}$ is the
block diagonal matrix whose successive diagonal blocks are $P_1,
\dots , P_k$.)  }

\begin{prop}\label{prop_redrep}  A formal series
 $F \in \mathcal{F}(\ca)$ corresponds to
a linearly representable shift-invariant probability measure $\mu$
on $\Omega_+(\ca)$ if and only if $F$ has a linear representation
$(x, \phi, y)$ with  $P=\sum_{a\in \ca}\phi(a)$ a stochastic matrix,
$y$ a column vector of all 1's, and $xP=x$. Moreover, in this case
the vector $x$ can be chosen to be positive with the matrix $P$ a
direct sum of irreducible stochastic matrices.
\end{prop}
\begin{proof}
It is straightforward to check that any $(x, \phi, y)$ of the
specified form linearly represents a shift-invariant measure.
Conversely, given a linear representation $(x, \phi, y)$ as in
Definition \ref{def_linrep} of a shift-invariant probability measure
$\mu$, define $P=\sum_{a \in \ca} \phi(a)$ and note that, by
induction, for all $w \in \ca^*,
\mu(\cc_0(w))=x\phi(w){P}^ky=x{P}^k\phi(w)y$ for all natural numbers
$k$.

Next, one shows that it is possible to reduce to a linear
representation $(x,\phi,y)$ of $\mu$ such that each entry of $x$ and
$y$ is nonzero, and, with ${P}$ defined as $P=\sum_{a\epsilon
\ca}\phi(a)$, $xP=x$ and $Py=y$. This requires some care. If indices
corresponding to 0 entries in $x$ or $y$, or to 0 rows or columns in
$P$, are jettisoned nonchalantly, the resulting new $\phi$ may no
longer be a morphism.

\begin{defn}
A triple $(x',\phi',y')$ is obtained from $(x,\phi ,y)$ by {\em
deleting a set $I$ of indices} if the following holds: the indices
for $(x,\phi ,y)$ are the disjoint union of the set $I$ and the
indices for $(x',\phi',y')$; and for every symbol $a$ and all
indices $i,j$ not in $I$ we have $x'_i=x_i, y'_i=y_i$ and
$\phi'(a)(i,j)=\phi(a)(i,j)$. Then we let $\phi '$ also denote the
morphism
 determined by the map on generators
$a\mapsto \phi'(a)$.
\end{defn}

First, suppose that $j$ is an index such that column $j$ of $P$
(and therefore column $j$ of  every $\phi (a):=M_a$) is zero. By
shift
invariance of the measure, $(xP,\phi ,y)$ is still a representation,
so we may assume without loss of generality that $x_j=0$. Let
$(x',\phi ',y)$ be obtained from $(x,\phi ,y)$ by deleting the index
$j$.
We claim that $(x',\phi ',y)$ still gives a linear representation
of $\mu$. This is because for any word $a_1\dots a_m$,
the difference
$[x\phi (a_1)\cdots \phi (a_m)y]
-[x'\phi' (a_1)\cdots \phi' (a_m)y']$ is a sum of terms of the form
 \be
x(i_0)M_{a_1}(i_0,i_1)M_{a_2}(i_1,i_2)\cdots
M_{a_m}(i_{m-1},i_m)y(i_m)
 \en
 in which at least one index
$i_t$ equals $j$. If $i_0=j$, then $x(i_0)=0$;
if $i_t=j$ with $t>0$, then
$M_{a_t}(i_{t-1},i_t)=0$. In either case, the product is zero.

By the analogous argument involving $y$ rather than $x$, we may
pass to a new representation by deleting the index of any zero
row of $P$. We repeat until we arrive at a representation in which
no row or column of $P$ is zero.

An {\em irreducible component} of $P$ is a maximal principal
submatrix $C$ which is an irreducible matrix. $C$ is an {\em initial
component} if
 for every index $j$ of a column through $C$, $P(i,j)>0$ implies
that $(i,j)$ indexes an entry of $C$. $C$ is a {\em terminal
component} if
 for every index $i$ of a row through $C$, $P(i,j)>0$ implies
that $(i,j)$ indexes an entry of $C$.

Now suppose that $\mathcal I$ is the index set of an initial
irreducible component of $P$, and $x(i)=0$ for every $i$ in
$\mathcal I$. Define $(x',\phi ',y)$ by deleting the index set
$\mathcal I$. By an argument very similar to the argument for
deleting the index of a zero column, the triple $(x',\phi ',y')$
still gives a linear representation of $\mu$. Similarly, if
$\mathcal J$ is the index set of a terminal irreducible component of
$P$, and $y(j)=0$ for every $j$ in $\mathcal J$, we may pass to a
new representation by deleting the index set $\mathcal J$.

Iterating these moves, we arrive at a representation for which $P$
has no zero row and no zero column; every initial component has an
index $i$ with $x(i)>0$; and every terminal component has an index
$j$ with $y(j)>0$. We now claim that for this representation the set
of matrices $\{P^n\}$ is bounded. Suppose not. Then there is a pair
of indices $i,j$ for which the entries $P^n(i,j)$ are unbounded.
There is some initial component
index $i_0$, and some $k\geq 0$, such that
$x(i_0)>0$ and $P^k(i_0,i)>0$. Likewise there is a terminal
component index
$j_0$ and an $m\geq 0$ such that $y(j_0)>0$ and $P^m(j,j_0)>0$.
Appealing to shift invariance of $\mu$, for all $n>0$ we have
 \be
1=xP^{n+k+m}y \geq x(i_0)P^k(i_0,i)P^n(i,j)P^m(j,j_0)y(j_0),
 \en
which is a contradiction to the unboundedness of the entries
$P^n(i,j)$. This proves the family of matrices $P_n$ is bounded.

Next let $Q_n$ be the Ces\`{a}ro sum, $(1/n)(P + ... + P^n)$. Let
$Q$ be a limit of  a subsequence of the bounded sequence $\{Q_n\}$.
Then $PQ=Q=QP$; $xQ$ and $Qy$ are {}{fixed vectors} of $P$; and
$(xQ,\phi ,Qy)$ is a linear representation of $\mu$. It could be
that $xQ$ vanishes on all indices through some initial component, or
that $Qy$ vanishes on all indices through some terminal component.
In this case we simply cycle through our reductions until finally
arriving a linear representation $(x,\phi ,y)$ of $\mu$ such that
{}{$xP=x$;  $Py=y$; the set of matrices $\{P^n\}$ is bounded;
 $P$ has no zero row or column; $x$ does not vanish
on all indices of any initial component; and $y$ does not vanish on
all indices of any terminal component.
}

{}{ If $C$ is an initial component of $P$, then the restriction of
$x$ to the indices of $C$ is a nontrivial fixed vector of $C$. Thus
this restriction is positive, and the
 spectral radius of $C$ is at least 1. The spectral radius
of $C$ must then be exactly 1, because the set $\{P^n\}$ is bounded. }

{}{ We are almost done. Suppose $P$ is not the direct sum of
irreducible matrices. Then there must be an initial component with
index set $\mathcal I$ and a terminal component with index set
$\mathcal J \neq \mathcal I$, with some $i\in \mathcal I$, $j\in
\mathcal J$ and $m$ minimal in $\mathbb N$ such that $P^m(i,j)>0$.
Because $\mathcal I$ indexes an initial component, for any $k\in
\mathbb N$ we have that $(xP^k)_i$ is the sum of the terms
$x_{i_0}P(i_0,i_1)\cdots P(i_{k-1},i)$ such that $i_t \in \mathcal
I$, $0\leq t \leq k-1$. Because $\mathcal J$ indexes an terminal
component, for any $k\in \mathbb N$ we have that $(P^ky)_j$ is the
sum of the terms $P(j,i_1)\cdots P(i_{k-1},i_k)y(i_k)$ such that
$i_t \in \mathcal J$, $1 \leq t \leq k$. Because $\mathcal I \neq
\mathcal J$, by the minimality of $m$ we have for all $n \in \mathbb
N$ that } \be xy =xP^{m+n}y \geq
\sum_{k=0}^n(xP^k)_iP^m(i,j)(P^{n-k}y)_j =(n+1)x_iP^m(i,j)y_j \ ,
\en a contradiction.

Consequently, $P$ is now a direct sum of irreducible matrices,
each of which has spectral radius 1. The eigenvectors $x,y$
are now positive. Let $D$ be the diagonal
matrix with $D(i,i)=y(i)$. Define $(x', \phi',y) = (xD,D^{-1}\phi
D,D^{-1}y)$. Then $(x', \phi',y)$ is the linear representation
satisfying all the conditions of the theorem.

\end{proof}

\begin{ex}\label{nononinv}
The conclusion of the Proposition does not follow without the
hypothesis of stationarity: there need {\em not} be any linear
representation with positive vectors $x,y$, and there need not be
any linear representation in which the nonnegative vectors $x,y$ are
fixed vectors of $P$.
 For example, consider the nonstationary Markov measure $\mu$
on two states $a,b$ with initial vector $p=(1,0)$ and transition
matrix
 \be
 T=
\left(
 \begin{array}{ccc}
    1/2 & 1/2\\
    0 &  1
     \end{array}  \right)
= \left(
 \begin{array}{ccc}
    1/2 & 1/2\\
    0 &  0
     \end{array}  \right)
+ \left(
 \begin{array}{ccc}
    0 & 0 \\
    0 &  1
     \end{array}  \right)
=N_a + N_b
 \ .    \en
If $q$ is the column vector $(1,1)^\text{tr}$, then $p,N_a,N_b,q$
generate a linear representation of $\mu$, e.g. $1=\mu(\cc_0(a)) =
pN_aq$, and
  $(1/2)^k=\mu(\cc_0(a^kb^m)) = p(N_a)^k(N_b)^mq$ when $k,m>0$ .

Now suppose that there is a linear representation of $\mu$
generated by positive vectors $x,y$ and nonnegative
matrices $M_a,M_b$. Then
 \be
 \begin{aligned}
 1&=\mu(\cc_0(a))=xM_ay ,\\
 0&=\mu(\cc_0(b))=xM_by .
 \end{aligned}
 \en
 From the second of these equations, $M_b=0$, since $x>0$ and $y>0$.
 But this contradicts
$
 0 < \mu(\cc_0(ab)) = xM_aM_by
$.

Next suppose there is a linear representation for which  $x,y$ could
be chosen eigenvectors of $P=M_a + M_b$ (necessarily with eigenvalue
1, since $xP^ny=1$ for all $n>0$). Then \be \frac 12 =
\mu(\cc_0(ab))= xM_aM_by \leq xPM_by = xM_by = \mu(\cc_0(b))=0 ,\en
which is a contradiction.

\end{ex}

\subsection{Equivalent characterizations of hidden Markov measures}

\subsubsection{Sofic measures---formal series approach}

The semiring $\F$ of formal series on the alphabet $\ca$ is an \em
${\R}_{+}$-module\em~ in a natural way. On this module we have a
{\em {(linear)} action of $\ca^{*}$} defined as follows:

For $F \in \mathcal{F}(\ca)$ and $w\in \ca^{*}$, define
$(w,F)\rightarrow w^{-1}F$ by
\begin{displaymath}
(w^{-1}F)(v)=F(wv) \textrm{ for all } v\in \ca^{*}.
\end{displaymath}
Thus
$$w^{-1}F=\sum_{v \in \ca^{*}}F(wv)v.$$
If $F=u\in \ca^{*}$, then
$$(w^{-1}F)(v)=u(wv)=
\begin{cases}1 & \text{if $wv=u$}
\\0 & \text{if $wv\neq u.$}\end{cases}$$
Thus $w^{-1}u\neq 0$ if and only if $u=wv$ for some $v\in \ca^{*}$,
and then $w^{-1}u=v$ (in the sense that they are the same function
on $\ca^{*}$): $w^{-1}v$ erases $w$ from $v$ if $v$ has $w$ as a
prefix, otherwise $w^{-1}v$ gives 0. Note also that this is a \em
monoid action\em~ :
\be \label{monoidtag}
(vw)^{-1}F=w^{-1}(v^{-1}F) \ .
\en

\begin{defn}
A submodule $M$ of $\mathcal{F}(\ca)$ is called \em stable\em~ if
$w^{-1}F\in M$ for all $F\in M$, i.e. $w^{-1}M \subset M$, for all
$w\in \ca^{*}$.
\end{defn}

\begin{thm}\label{th_rationalequiv}
Let $\ca$ be a finite alphabet. For a formal series $F \in \mathcal
F_{{\R}_{+}} (\ca)$ that corresponds to a shift-invariant
probability measure $\nu$ in ${\Omega}_{+} (\ca)$, the following are
equivalent:
\begin{enumerate}
\item $F$ is linearly representable.
\item $F$ is a member of a stable finitely generated submodule of $\mathcal F_{{\R}_{+}}
(\ca)$.
\item $F$ is rational.
\item The measure $\nu$ is the image under a 1-block map of a shift-invariant
1-step Markov probability measure $\mu$.
\end{enumerate}
In the {latter} case, the measure $\nu$ is ergodic if and only if it
is possible to choose $\mu$ ergodic.
\end{thm}
In the next few sections we sketch the proof of this theorem

\subsubsection{Proof that a series is linearly representable if and only if
it is a member of a stable finitely generated submodule of
$\F$}\label{subsec_lr=fg}

Suppose that $F$ is linearly representable by $(x, \phi, y)$. For
each $i=1,2, \cdots ,n$ (where $n$ is the dimension of the
representation) and each $ w\in \ca^{*}$, define
$$F_{i}(w)=[\phi(w)y]_{i}.$$
 Let $M= \langle F_{1},\cdots, F_{n} \rangle$ be the span of the
$F_{i}$ with coefficients in $\R_{+}$, which is a submodule of
$\mathcal{F}(\ca)$.  Since
$$F(w)=x\phi(w)y=\sum_{i=1}^{n}
x_{i}[\phi(w)y]_{i}= \sum_{i=1}^{n}x_{i}F_{i}(w),$$ we have that $F=
\sum_{i=1}^{n} x_{i}F_{i}$, which means $F\in M$.

We next show that $M$ is stable. Let $w\in \ca^{*}$.  Then for $u\in
\ca^{*}$,
\begin{eqnarray*}\begin{gathered}(w^{-1}F_{i})(u)=F_{i}(wu)=[\phi(wu)y]_{i}
=[\phi(w)\phi(u)y]_{i}
\\=\sum_{j=1}^{n} \phi(w)_{ij}[\phi(u)y]_{j}=\sum_{j=1}^{n}
\phi(w)_{ij}F_{j}(u).\end{gathered}\end{eqnarray*}

Since $\phi(w)_{ij} \in \R_{+}$, we have $\sum_{j=1}^{n}
\phi(w)_{ij}F_{j}(u) \in M,$ so
$$w^{-1}F_{i}=\sum_{j=1}^{n} x_{i}\phi(w)_{ij}F_{j}\in
\langle F_{1},...F_{n} \rangle =M.$$

Conversely, let $M$ be a stable finitely generated left submodule,
and assume that $F\in \langle F_{1},\cdots,F_{n}\rangle =M.$ Then
there are $x_1, \cdots, x_n \in {\R}_{+}$ such that
$F=\sum_{i=1}^{n} x_{i}F_{i}$.  Since $M$ is stable, for each $a\in
\ca$ and each $i=1,2,\cdots,n$, we have that $a^{-1}F_{i}\in \langle
F_{1},...F_{n} \rangle$. So there exist $c_{ij}\in \R_{+},
j=1,2,\cdots, n,$ such that $a^{-1}F_{i}=\sum_{j=1}^{n}
c_{ij}F_{j}$.
 Define $\phi(a)_{ij}=c_{ij}$ for $i,j=1,2,\cdots ,n.$
Note by linearity that for any nonnegative row
vector $(t_1, \dots ,t_n)$
we have
\be \label{lineartag}
a^{-1}(\sum_{i=1}^n t_iF_i) =
\sum_{j=1}^n\Big( (t_1, \dots , t_n)\phi(a)\Big)_j F_j \ .
\en
Extend $\phi$ to a monoid morphism $\phi: \ca^{*}\rightarrow
\R_{+}^{n \times n}$ by defining $\phi(a_1 \cdots a_n) = \phi(a_1)
\cdots \phi(a_n)$. Because the action of $\ca^{*}$ on
$\mathcal{F}(\ca)$ satisfies the monoidal condition
(\ref{monoidtag}), we have
from (\ref{lineartag}) that
for any  $w=a_1a_2\cdots a_n \in \ca^{*}$,
\begin{align*}
w^{-1}(\sum_{i=1}^n t_iF_i) &=
(a_1\cdots a_n)^{-1}(\sum_{i=1}^n t_iF_i)
=(a_n^{-1} \cdots  (a_1^{-1}\sum_{i=1}^n t_iF_i)\cdots ) \\
&= \sum_j \Big((t_1,\dots ,t_n) \phi (a_1)\cdots \phi
(a_n)\Big)_jF_j
= \sum_j \Big((t_1,\dots ,t_n) \phi (w)\Big)_jF_j \ .
\end{align*}
 Define the column vector $y$ by
 $y_{j}=F_{j}(1)$ for $j=1,2,\cdots,n$
 and let $x$ be the row vector $(x_1, \dots ,x_n)$.
Then
\be
F(w)= w^{-1}F(1) = \Bigg(\sum_j\Big(x\phi (w)\Big)_jF_j\Bigg)(1)
= \sum_j\Big(x\phi (w)\Big)_jF_j(1) =x\phi (w) y \ ,
\en
 showing that  $(x,\phi, y)$ is a linear representation
for $F$.

\subsubsection{Proof that a formal series is linearly representable if and only if it is
rational}\label{subsec_lr=rat}

This equivalence is from \cite{Kleene1956, Schutzenberger1961}.
Recall that a series is rational if and only if it is in the closure
of the polynomials under the rational operations $+$ $($union$)$,
$\cdot$ $($concatenation$)$, $*$, and multiplication by elements of
$\R_{+}$.

First we prove by a series of steps that every rational series $F$
is linearly representable.
\begin{prop}
Every polynomial is linearly representable.
\end{prop}
\begin{proof}
If $w \in \ca$ and $|w|$ is greater than the degree of the
polynomial $F$, then $w^{-1}\equiv 0$.  Let $S=\{w^{-1}F: w\in
\ca^{*}\}.$ Then $S$ is finite and stable, hence $S$ spans a
finitely generated stable submodule $M$ to which $F$ belongs.  (Take
${\epsilon}^{-1}F=F$). By Section \ref{subsec_lr=fg}, $F$ is
linearly representable.
\end{proof}
The next observation follows immediately from the definition of
stability. {The proof of the Lemma is included for practice.}

\begin{prop}
If $F_{1}$ and $F_{2}$ are in stable finitely generated submodules
of $\mathcal{F}(\ca)$ and $t \in \R_{+}$, then $(F_{1}+F_{2})$ and
$(tF_{1})$ are in stable finitely generated submodules
of $\mathcal{F}(\ca)$.
\end{prop}
\begin{lem}
 For $F,G \in \F$ and $a \in \ca$, $a^{-1}(FG) =
(a^{-1}F)G+F(\epsilon)a^{-1}G$.
    \end{lem}
\begin{proof} For any $w \in \ca^{*}$,
\begin{equation}
\begin{split}
(a^{-1}(FG))(w) &= (FG)(aw)=\sum_{uv=aw}F(u)G(v)\\
    &=F(\epsilon)G(aw)+ \sum_{u'v'=w}F(au')G(v')\\
    &=F(\epsilon)G(aw)+\sum_{u'v'=w}(a^{-1}F)(u')G(v')\\
    &=F(\epsilon)(a^{-1}G)(w)+((a^{-1}F)(G))(w).
\end{split}
\end{equation}
\end{proof}
\begin{prop}
Suppose that for $i=1,2$, $F_{i}\in M_{i}$, where each $M_{i}$ is a
stable, finitely generated submodule. Let $M= M_{1}F_{2}+M_{2}.$
Then $M$ is finitely generated and stable
{and contains $F_1F_2$.}
\end{prop}
\begin{proof} {
The facts that $F_1F_2 \in M$ and $M$ is finitely generated are}
immediate. The proof that $M$ is stable is a consequence of the
Lemma. For if $f_1F_2+f_2$ is an element of $M$ and $a \in \ca$,
then \be
a^{-1}(f_{1}F_{2}+f_2)=(a^{-1}f_{1})F_{2}+f_{1}(\epsilon)(a^{-1}F_{2})+a^{-1}f_2.
\en Note that $a^{-1}f_{1}\in M_{1}$ and $a^{-1}f_{2}$, $a^{-1}F_{2}
\in M_{2}$.  Thus $f_{1}(\epsilon)(a^{-1}F_{2})+f_2 \in M_{2}$, so
we conclude that
$M$ is stable.\\
\end{proof}
\begin{lem}
If $F$ is proper (that is $F_{1}(\epsilon) = 0$) and $a \in \ca$,
then $a^{-1}(F^{*}) = (a^{-1}F)F^{*}$.
\end{lem}
\begin{proof}
 Recall that $F_{1}^{*}=\sum_{n \geq 0}F_{1}^{n}$. Thus
 $a^{-1}(F^{*}) = a^{-1}(1+FF^{*}) =
a^{-1}(\epsilon+FF^{*}) = a^{-1}\epsilon +
(a^{-1}F)F^{*}+F(\epsilon)a^{-1}(F^{*}).$

Because $(a^{-1} \epsilon)(w) = \epsilon(aw) =0$ for all $w \in
\ca^{*}$ and $F(\epsilon)=0$, we get that $a^{-1}F^{*} =
(a^{-1}F)F^{*}$.
\end{proof}
\begin{prop}
Suppose $M_{1}$ is finitely generated and stable, and that $F_{1}
\in M_{1}$ is proper. Then $F_{1}^{*}$ is in a finitely generated
stable submodule.
\end{prop}
\begin{proof}
Define $M=\R_{+} + M_{1}F_{1}^{*}$.  We have
$$F_{1}^{*}=1+ \sum_{n \geq 1}F_{1}^{n}=(1+
F_1F_{1}^{*})\in M.$$  Also $M$ is finitely generated (by $1$ and
the $f_iF_1^{*}$ if the $f_i$ generate $M_1$).

To show that $M$ is stable, suppose that $t \in {\R}_{+}$ and $a \in
\ca$. Then for any $u \in \ca^{*}$ we have $(a^{-1} t)(u)=t(au)=0$,
so $a^{-1} t=0 \in {\R}_{+}$. And for any $f_1 \in M_1$ and $a \in
\ca$,
$a^{-1}(f_1F_1^{*})=(a^{-1}f_1)F_1^{*}+f_1(\epsilon)a^{-1}(F_1^{*})$.
Since $M_1$ is stable, $a^{-1}f_1 \in
M_1$ and the first term is in $M_1 F_1^{*}$. By the Lemma, the
second term is $f_1(\epsilon)(a^{-1}F_1)F_1^{*}$, which is again in
$M_1 F_1^{*}$.
\end{proof}

These observations show that if $F$ is rational, then $F$ lies in a
finitely generated stable submodule, so by Section
\ref{subsec_lr=fg} $F$ is linearly representable.

Now we turn our attention to proving the statement in the title of
this section in the other direction. So assume that $F \in \F$ is
linearly representable. Then $F(w) = x  \phi (w) y$ for all $w \in
\ca$ for some $(x,\phi,y)$. Consider the semiring of formal series
$\mathcal{F}_{K}(\ca) = K^{\ca^{*}}$, where $K$ is the semiring
$\R_{+}^{n \times n}$ of $n \times n$ nonnegative real matrices and
$n$ is the dimension of the representation.
Let $D = \sum_{a \in \ca} \phi(a) a \in \mathcal{F}_{K}(\ca).$  The
 series $D$ is proper, so we can form
 \be D^{*} = \sum_{h \geq 0}D^{h}
        = \sum_{h \geq 0} \Bigl( \sum_{a \in \ca} \phi(a) a \Bigr)^{h}
        =\sum_{h \geq 0}\Bigl(\sum_{w \in \ca^{h}} \phi(w)w \Bigr)
        =\sum_{w \in \ca}\phi(w) w.
\en This series $D^{*}$ is a rational element of
$\mathcal{F}_{K}(\ca),$  since we started with a polynomial and
formed its *. By Lemma \ref{lem_rationallemma} below, each entry
$(D^{*})_{ij}$ is rational in $\mathcal{F}_{\R_{+}}(\ca)$.

With $D$ and $D^*$ now defined, we have that \be F(w) = x\phi(w) y
       =\sum_{i,j}x_{i}\phi(w)_{ij}y_{j}
       =\sum_{i,j}x_{i}D^{*}(w)_{ij}y_{j},
\en and each $D^{*}(w)_{ij}$ is a rational series applied to $w$.
Thus $F(w)$ is a finite linear combination of rational series
$D_{ij}^{*}$ applied to $w$ and hence  is rational.

\begin{lem}\label{lem_rationallemma}
Suppose $D$ is an $n\times n$ matrix whose entries are proper
rational formal series (e.g., polynomials). Then the entries of
$D^*$ are also rational.
\end{lem}
\begin{proof}
We use induction on $n$. The case $n=1$ is trivial. Suppose the
lemma holds for $n-1$, and $D$ is $n\times n$ with block form $D=
\begin{pmatrix}
a & u \\
v & Y
\end{pmatrix}$,
with $a$ a rational series. The entries of $D$ can be thought of as
labels on a directed graph; a path in the graph has a label which is
the product of the labels of its edges; and then $D^*(i,j)$
represents the sum of  the labels of all paths from $i$ to $j$
(interpret the term ``1'' in $D(i,i)$ as the label of a path of
length zero).
 With
this view, one can see that $D^*=
\begin{pmatrix}
b & w \\
x & Z
\end{pmatrix}$, where
\begin{enumerate}
\item
$b=(a+uY^*v)^*$ ,
\item
$Z= (Y+va^*u)^* $ ,
\item
$w=buY^*$ ,
\item
$x=Y^*vb$ .
\end{enumerate}
Now $Y^*$ and $Z$ have rational entries by the induction hypothesis,
and consequently all entries of $D^*$ are rational.
\end{proof}

\subsubsection{{Linearly representable series correspond to sofic
measures}}

{} {The (topological) support of a measure is the smallest closed
set of full measure. Recall our convention (Sec. \ref{sec_hmms})
that Markov and sofic measures are ergodic with full support.}

\begin{thm}[\cite{Furstenberg1960,HanselPerrin1989,Heller1965}]\label{rationaltheorem}
A shift-invariant probability measure $\nu$ on $\Omega_{+}(\ca)$
corresponds to a
 linearly representable (equivalently, rational)
 formal series $F =
F_{\nu} \in \mathcal{F}_{\R_{+}}(\ca)$ if and only if it is {}{ a
convex combination of measures which (restricted to their supports)
are sofic measures. }
 {}{ Moreover,
if $(x,\phi , y)$ is a representation of $F_{\nu}$ such that $x$ and
$y$ are positive and the matrix $\sum_{i\in B}\phi (i)$ is
irreducible, then $\nu$ is a sofic measure. }
\end{thm}

\begin{proof}
 {
 Suppose that $\nu$ is the image under a 1-block map
(determined by a map $\pi : \ca \rightarrow \cb$ between the
alphabets) of a 1-step Markov measure $\mu$. Then $\nu$ is linearly
representable by the construction in Example \ref{ex_linrepofsofic}.
}

Alternatively, if $F_{\mu}$ is represented by $(x, \phi , y)$ then
for each $w \in \ca^{*}$ we have
 \be
F_{\mu}(w)
= \sum_{i,j}x_i \phi(w)_{ij}y_j
= \sum_{i,j}x_i \Big([\sum_{a \in \ca} \phi(a)a]^{*}(w)\Big)_{ij}y_j.
 \en
 For  $u \in {\mathcal B^*}$ define
 \be
F_{\nu} (u)
=\sum_{i,j} x_i \Bigg( \Big[ \sum_{b \in \mathcal B }
 (\sum_{a\in \ca,  \phi (a)=b} \phi(a))b\Big]^{*}(u)\Bigg)_{ij}y_j
\en
to see that $F_{\nu}$ is a linear combination of rational series
and to see its linear representation.

 Conversely, suppose that $\nu$ corresponds to a rational (and
hence linearly representable) formal series $F=F_{\nu} \in
\mathcal{F}_{\R_{+}}(\mathcal B)$ with dimension $n$.  Let $(x,
\phi, y)$ represent $F$. {}{ To indicate an ordering of the alphabet
$\mathcal B$, we use notation $\mathcal B=\{1,2, \dots , k\}$ and
$\phi (i) = P_i$. First assume that the $n\times n$ matrix $P$ is
irreducible and the vectors $x$ and $y$ are positive. We will
construct a Markov measure $\mu$ and a 1-block map $\pi$ such that
$\nu  = \pi \mu$. }

{}{ Applying the standard stochasticization trick as in the last
paragraph of the proof of Proposition \ref{prop_redrep}, we may
assume that the irreducible matrix $P$ is stochastic, every entry of
$y$ is 1, and $x$ is stochastic. Define matrices with block forms,
\begin{align*}
 M =
\begin{pmatrix}
P_1 & P_2 & \cdots &P_k \\
P_1 & P_2 & \cdots &P_k \\
\cdots &\cdots & \cdots &\cdots  \\
P_1 & P_2 & \cdots &P_k
\end{pmatrix} \ , \qquad
R& =\begin{pmatrix}
I \\
I \\
\cdots \\
I
\end{pmatrix} \ , \\
C=\begin{pmatrix}
P_1 & P_2 &
\cdots & P_k
\end{pmatrix} \ , \quad
M_i &=
\begin{pmatrix}
0 & \cdots & P_i & \cdots & 0 \\
0 &  \cdots & P_i & \cdots & 0 \\
\cdots  &  \cdots & \cdots & \cdots &\cdots  \\
0 &  \cdots & P_i & \cdots & 0
\end{pmatrix} \
\end{align*}
where each $P_i$ is $n\times n$; $R$ is $nk\times k$;
$I$ is the $n\times n$ identity matrix; $C$ and the $M_i$ are
 $nk \times nk$; and  $M_i$ is zero except in the $i$'th block
column, where it is $RP_i$.
The matrix $M$ is stochastic, but
it can have zero columns. (We thank Uijin Jung for
pointing this out.)  Let $M'$ be the largest principal
submatrix of $M$ with no zero column or row.
}

{}{ We have a strong shift equivalence $M=RC$, $P=CR$,  and it then
follows from the irreducibility of $P$ that  $M'$ is irreducible.
Therefore, there is a unique left stochastic fixed vector $X$ for
$M$. Let $Y$ be the $nk\times 1$ column vector with every entry 1.
We have $MR=RP$, and consequently $XR=x$. Also, $M_iR=RP_i$ for each
$i$. So, for any word $i_1 \cdots i_j$, we have
\begin{align*}
xP_{i_1}\cdots P_{i_j}y &=
XRP_{i_1}\cdots P_{i_j}y \\
&=XM_{i_1}\cdots M_{i_j}Ry =
XM_{i_1}\cdots M_{i_j}Y \ .
\end{align*}
This shows that $(X,\Phi , Y)$ is also a representation of
$F_{\nu}$, where $\Phi(i)=M_i$. Let $X', \Phi'(i)=M_i', Y'$ be the
restrictions of $X, \Phi (i),Y$ to the vectors/matrices on the
indices of $M'$. Then  $(X',\Phi' , Y')$ is also a representation of
$F_{\nu}$. Let $A'$ be the $0,1$ matrix of size matching $M'$ whose
zero entries are the same as for $M'$. Then $(X',M',Y')$ defines an
ergodic Markov measure $\mu$ on $\Omega_{A'}$ and there is a 1-block
code $\pi$ such that $\pi \mu = \nu$.
 Explicitly, $\pi$ is the restriction of the code which
sends $\{ 1,2,\dots n\}$ to 1; $\{n+ 1,n+2,\dots 2n\}$ to 2; and so
on.
Thus $\nu$ is a sofic measure.
}

{}{ Now, for the representation $(x,\phi ,y)$ of $F_{\nu}$, we drop
the assumption that the matrix $P$ is irreducible. However, by
Proposition \ref{prop_redrep}, without loss of generality we may
assume that $P$ is the direct sum of irreducible stochastic matrices
$P^{(j)}$;  $x$ is a positive stochastic left fixed
 vector of $P$; and $y$ is the column vector
with every entry 1. Restricted to the indices through  $P^{(j)}$,
$x$ is a fixed vector of  $P^{(j)}$ and therefore is a multiple
$c_jx^{(j)}$ of the stochastic left fixed vector  $x^{(j)}$ of
$P^{(j)}$. Note, $\sum_jc_j=1$. If  $y^{(j)}$ denotes the column
vector with every entry 1 such that
 $P^{(j)}y^{(j)}=y^{(j)}$, then
\[
(x,\phi ,y ) = \sum_j c_j(x_j, P^{(j)}, y^{(j)}) \ .
\]
If follows from the irreducible case that $\mu$ is a convex
combination
of sofic measures.
}
 \end{proof}

\subsection{Sofic measures---Furstenberg's
approach}\label{sec_furst}

Below we are extracting from \cite[Secs. 18--19]{Furstenberg1960}
only what we need to describe Furstenberg's approach to the
identification of sofic measures and compare it to the others. This
leaves out a lot. We follow Furstenberg's notation, apart from
change of symbols, except that we refer to shift-invariant measures
as well as finite-state stationary processses.

Furstenberg begins with the following definition.

\begin{defn} \cite[Definition 18.1]{Furstenberg1960}
A {\it stochastic semigroup of order r} is a semigroup $\mathcal S$
having an identity $e$ (i.e., a monoid), together with a set of $r$
elements
  {$\ca=
\{ a_1, a_2, \dots , a_r\}$
 }
generating $\mathcal S$, and a real-valued function $F$ defined on
$S$ satisfying
\begin{enumerate}
\item
$F(e)=1$,
\item
$F(s )\geq 0$ for each $s \in S$ and $F(a_i) > 0, \ i=1,2,\dots ,r$
,
\item
$\Sigma_{i=1}^r F(a_is )= \Sigma_{i=1}^r F(s a_i)= F(s )$ for each
$s \in S$.
\end{enumerate}
\end{defn}

Given a subshift $X$ on an alphabet $\{a_1, a_2, \dots , a_r\}$ with
shift-invariant Borel probability $\mu $ and $\mu (a_i)>0$ for every
$i$,  let $S$ be the free semigroup of all formal products of the
$a_i$, with the empty product taken as the identity $e$. Define $F$
on $S$ by $F(e)=1$ and $F(a_{i_1}a_{i_2}\dots  a_{i_k})= \mu
(\cc_0(a_{i_1}a_{i_2}\dots   a_{i_k}))$. Clearly the triple $(\{a_1,
a_2, \dots , a_r\},S,F)$ is a stochastic semigroup, which we denote
$S(X)$.

Conversely,  any stochastic semigroup $(\{ a_1, a_2, \dots ,
a_r\},S,F)$ determines a unique shift-invariant Borel probability
$\mu$ for which $F(a_{i_1}a_{i_2}\dots  a_{i_k})=
\mu(\cc_0((a_{i_1}a_{i_2}\dots   a_{i_k})))$ for all
$a_{i_1}a_{i_2}\dots a_{i_k}$. We denote by $X(S)$ this finite-state
stationary process (equivalently the full shift on $r$ symbols with
invariant measure $\mu$). Two stochastic semigroups are called {\it
equivalent} if they define the same finite-state stationary process
modulo a bijection of their alphabets. A {\em cone} in a linear
space is a subset closed under addition and multiplication by
positive real numbers \cite[Sec. 15.1]{Furstenberg1960}.

\begin{defn}\label{defn_linearsemigp} \cite[Definition 19.1]{Furstenberg1960}
Let $D$ be a linear space, $D^*$ its dual, and let $\cc$ be a cone
in $D$ such that for all $x,y$ in $D$, if $x+\lambda y \in \cc$ for
all real $\lambda$, then $y=0$. Let $\theta\in \cc$ and $\theta^*\in
D^*$, and suppose that  $\theta^*$ is nonnegative on $\cc$. A {\it
linear stochastic semigroup } $S$ on $(\cc, \theta,\theta^*)$ is a
stochastic semigroup $(\{a_1,\dots,a_r\},S,F)$ whose elements are
linear transformations {\em from $\cc$ to $\cc$} satisfying
\begin{enumerate}
\item
$\sum a_i \theta = \theta$;
\item
$\sum a_i^* \theta^* = \theta^*$ (where $L^*$ denotes the
transformation of $D^*$ adjoint to a transformation $L$ of $D$);
\item
$F(s)= (\theta^*,s\theta )$ for  $s \in S$, where $(\cdot , \cdot )$
denotes the dual pairing of $D^*$ and $D$;
\item
$(\theta^*,a_i \theta) > 0$, $\ i=1,2, \dots , r$.
\end{enumerate}
$(S,D,\cc,\theta,\theta^*)$ was called {\em finite dimensional} by
Furstenberg if there is $m \in \Bn$ such that $D=\Br^m$, $\cc$ is
the cone of vectors in $\Br^m$ with all entries nonnegative, and
each element of $S$ is an $m \times m$ matrix with nonnegative
entries.
\end{defn}

A semigroup $S$ of transformations satisfying $(1)$ to $(4)$ does
define a stochastic semigroup if $(\theta^*,\theta) = 1$.

\begin{thm}\cite[Theorem 19.1]{Furstenberg1960}
Every stochastic semigroup $S$ is equivalent to some linear
stochastic semigroup.
\end{thm}
\begin{proof}
Let $A_0(S)$ be the real semigroup algebra of $S$, i.e., the real
vector space with basis $S$ and multiplication determined by the
semigroup multiplication in $S$ and the distributive property,

\begin{equation}
\Big(\sum \alpha_{s }s\Big)\Big(\sum \beta_{t }t \Big)= \sum
\alpha_{s }\beta_{t} st .
\end{equation}
(Each sum above has finitely many terms.)

If $S$ is the free monoid generated by $r$ symbols, then
  {
  $A_0(S)$
   is
isomorphic to the set
 $\wp_{\Br}(\ca)$ of real-valued polynomials, i.e. finitely supported
 formal series $\ca^* \to \Br$ (see Definition \ref{defn_poly}).}

 Extend $F$ from $S$ to a linear functional on $A_0(S)$, i.e.
$F(\sum \alpha_{s }s)= \sum \alpha_{s }F(s)$. Define $I=\{u\in
A_0(S): F(u)=0\}$, an ideal in
 $A_0(S)$, and the algebra
 {
 $A=A(S)=A_0(S)/I$.
 }
Define the element $\tau = a_1 + a_2 + \dots +a_r$ in $A(S)$ (here
$a_i$ abbreviates $a_i + I$) and set $D= A/A(e-\tau )$.

The elements of $A$ and in particular those of $S$ operate on $D$ by
left multiplication. Let $a_i'$ denote the operator induced by left
multiplication by $a_i \in S$. Take $V$ to be the image in $D$ of
the set of elements of $A$ that can be represented as positive
linear combinations of elements in $S$. Denote by $\overline u$ the
image in $D$ of an element $u$ in $A$. Set $\theta = \overline e$
and let $\theta^*$ be the functional induced on $D$ by $F$ on $A$
($F$ vanishes on $A(e-\tau )$).

Then the four conditions in the definition of linear stochastic
semigroup are satisfied. This linear stochastic semigroup given by
\begin{equation} \label{furstenbergmodule}
(\{a'_1,\dots ,a'_r\}, D,V,\theta ,\theta^*)
\end{equation}
is equivalent to the given $S$ because $F(s ')= (\theta^*,s '\theta)
= F(s )$. (We will see later that this construction is closely
related to
 Heller's ``stochastic module'' construction.)
\end{proof}

Given a shift-invariant sofic measure on the set of two-sided
sequences on the alphabet $\{1, \dots,r\}$ which assigns positive
measure to each symbol,
 it is
possible to associate an explicit finite-dimensional linear
stochastic semigroup to
 $\mu$
  in the same way that we attached a linear representation in Example
 \ref{ex_linrepofsofic}.
  Here $\mu$ is
the image under some 1-block code $\pi$ of a Markov measure defined
from some $m\times m$  stochastic matrix $P$. For $1\leq i \leq r$,
let $P_i$ be the $m\times m$ matrix  such that $P_i(i',j') =
P(i',j')$ if $\pi (j')=i$ and otherwise $P_i(i',j') = 0$. Let
$\theta^* $ be a stochastic (probability) left fixed vector for
$P$ and let $\theta $ be the column vector with every entry $1$.
Let $C$ be the cone of all nonnegative vectors in $D= \mathbb R^m$.
If we identify $P_i$ with the symbol $i$, then these data give a
finite-dimensional linear stochastic semigroup equivalent to $S(X)$.
Along with this observation, Furstenberg established the converse.

 \begin{thm}\label{furstfinite}
 \cite[Theorem 19.2]{Furstenberg1960}
  A linear stochastic semigroup $S$ is
finite dimensional if and only if the stochastic process that it
determines is a 1-block factor of a 1-step stationary finite-state
Markov process.
\end{thm}
In the statement of Theorem \ref{furstfinite}, ``Markov'' does not
presume ergodic. The construction for the theorem is essentially the
one given in Theorem \ref{rationaltheorem},
with a simplification. Because of the definition of linear
stochastic semigroup (Definition \ref{defn_linearsemigp}),
Furstenberg can begin with $\theta^*, \theta$ actual  fixed vectors
of $P:= \sum_i P_i$. The triple $(P ,\theta^*, \theta )$ corresponds
to $(P,x,y)$ in Theorem \ref{prop_redrep}, where $x,y$ need not be
fixed vectors. Thus Furstenberg can reduce more quickly to the form
where $\theta^*$ and $ \theta $ are {\em positive} fixed vectors of
$P$. Note that ``finite dimensional'' in Theorem \ref{furstfinite}
means more than having the cone $C$ of the linear stochastic
semigroup generating a finite-dimensional space $D$: here $C$ is a
cone in $\mathbb R^m$ with exactly $m$ (in particular, finitely
many) extreme rays.

\subsection{Sofic measures---Heller's approach}

Repeating some problems already stated, but with some refinements,
here are the natural questions about sofic measures which we are
currently discussing, in subshift language.

\begin{prob}\label{prob_1}
Let $\pi: \Omega_A \to Y$ be a 1-block map from a shift of finite
type to a (sofic) subshift and let $\mu$ be a (fully supported)
1-step Markov measure on $\Omega_A$. When is $\pi\mu$ Markov? Can
one determine what the {\em order} (a $k$ such that the measure is
$k$-step Markov) of the image measure might be?
\end{prob}

\begin{prob}\label{prob_2}
Given a shift-invariant probability measure $\nu$ on a subshift $Y$,
when are there a shift of finite type $\Omega_A$, a factor map $\pi:
\Omega_A \to Y$, and a 1-step shift-invariant fully supported Markov
measure $\mu$ on $\Omega_A$ such that $\pi\mu=\nu$?
\end{prob}

\begin{prob}\label{prob_3}
If $\nu$ is a sofic measure, how can one explicitly construct Markov
measures of which $\nu$ is a factor? Are there procedures for
constructing Markov measures that map to $\nu$ which have a minimal
number of states or minimal entropy?
\end{prob}

Problem \ref{prob_1} was discussed in \cite{BurkeRosenblatt1958},
for the reversible case. Later complete solutions depend on Heller's
solution of Problem \ref{prob_2}, so we discuss that first.
Effective answers to the first part of Problem \ref{prob_3} are
given by Furstenberg and in the proof of Theorem
\ref{rationaltheorem}.

{ Problem \ref{prob_2} goes back at least to a 1959 paper of Gilbert
\cite{Gilbert1959}. Following Gilbert and Dharmadhikari
\cite{Dharma1963,Dharma1963-2,Dharma1964,Dharma1965}, Heller (1965)
created his stochastic module theory and within this gave a
characterization
 \cite{Heller1965,Heller1967}
of sofic measures (1965). We describe this next. }

\subsubsection{Stochastic module}\label{subsubsec_stochmod}

We describe the stochastic module machinery setup of Heller
\cite{Heller1965} (with some differences in notation).
Let $S=\{1,2,...,s\}$ be a finite state space for a stochastic
process. Let $A_S$ be the associative real algebra with free
generating set $S$.
An {\em $A_S$-module} is a real vector space $V$ on
which $A_S$ acts by linear transformations, such that for each $i
\in S$ there is a linear transformation $M_i:V\to V$ such that a
word $u_1...u_k$ sends $v \in V$ to
$M_{u_1}(M_{u_2}(...(M_{u_k}(v))..)$. We denote an $A_S$-module as
 $(\{M_i\},V)$ or for brevity just $\{M_i\}$,
where the $M_i$ are the associated generating
linear transformations $V \to V$ as above.

\begin{defn}\label{def_stochmod}
A {\em stochastic S-module} for a stochastic process with
state space $S$
 is a triple $(l,\{M_i\},r)$, where
$(\{M_i\},V)$ is an $A_S$-module, $r\in
V$, $l \in V^*$, and for every word $u=u_1...u_t$ on $S$ its
probability $\text{Prob}(u)=\text{Prob}(\cc_0(u))$ is given by
 \be
 \text{Prob}(u)=    l M_{u_1}M_{u_2}...M_{u_t}r.
 \en
 \end{defn}
Given an $A_S$-module $M$, an $l \in V^*$ and $r \in V$, a few
axioms are required to guarantee that they define a stochastic
process with state space $S$. Define
$\sigma=\sum\{a_i: a_i \in S\}$ and
denote by $\cc_S$ the cone of polynomials in $A_S$ with nonnegative
coefficients. Then the axioms are that
\begin{enumerate}
\item $lr=1$;
\item $l(\cc_Sr) \subset [0, \infty)$;
\item for all $f \in A_S$, $l(f(\sigma -1)\, r)=0$.
\end{enumerate}

\begin{ex}{\it A stochastic module for a sofic measure.}  \label{stochsofic}
As we saw in Section \ref{sec_furst}, this setup of a stochastic
module arises naturally when a 1-block map $\pi$ is applied to a
1-step Markov measure $\mu$ with state space $S$ given by an
$s\times s$ stochastic transition matrix
 $P$
and row probability vector $l$. For each $i \in S$, let $M_i$ be the
matrix whose $j$'th column equals column $j$ of $P$ if $\pi (j)=i$
and whose other columns are zero. The probability of an $S$-word
$u=u_1...u_t$ is
 $lM_{u_1}M_{u_2}...M_{u_t}r$, where $r$ is the vector of all
1's. With $V=\Br^s$, presented as column vectors, $(l,
\{M_i\},r)$ is a
stochastic module for the process given by $\pi\mu$.
\end{ex}

\subsubsection{The reduced stochastic
module}\label{subsubsec_redstochmod}
{ A stochastic module
$(l,(\{M_i\},V),r)$ is {\it reduced} if (i) $V$ is the smallest
invariant (under the operators $M_i$) vector space containing $r$
and (ii) $l$ annihilates no nonzero invariant subspace of $V$. }
Given a stochastic module $(l,\{M_i\},r)$ for a stochastic process,
with its operators $M_i$ operating on the real vector space $V$, a
smallest stochastic module $(l',\{M_i'\},r')$ describing the
stochastic process may be defined as follows. Let $R_1$ be the
cyclic submodule of $V$ generated by the action on $r$; let $L_1$ be
the cyclic submodule of $V^*$ generated by the (dual) action on $l$;
let $V'$ be $R_1$ modulo the subspace annihilated by $L_1$; for each
$i \in S$ let $M'_i$ be the (well defined) transformation of $V'$
induced by $M_i$; let $r',l'$ be the elements of $V'$ and
$(V')^\perp$ determined by $r,l$. Now $(l',M',r')$ is {\em the
reduced stochastic module} of the process.
 $V'$ is the subspace generated by the
action of the $M'_i$ on $r'$, and no nontrivial submodule of $V'$ is
annihilated by $l'$. The reduced stochastic module is still a
stochastic module for the original stochastic process.
We say ``the"
reduced stochastic module because any stochastic modules describing
the same stochastic process have isomorphic reduced stochastic
modules.

\subsubsection{Heller's answer to Problem \ref{prob_2}}
\label{subsubsec_hellersanswer}
 We give some preliminary
notation.
 A process is ``induced from a Markov
chain'' if its states are lumpings of states of a finite state
Markov process, that is, there is a 1-block code which sends the
associated Markov measure to the measure associated to the
stochastic process. Let $(A_S)_+$ be the subset of $A_S$ consisting
of linear combinations of words with all coefficients nonnegative. A
{\it cone} in a real vector space $V$ is a union of rays from the
origin. A convex cone $\mathcal C$ is {\it strongly convex} if it
contains no line through the origin. It is {\it polyhedral} if it is
the convex hull of finitely many
rays.

\begin{thm}\label{thm_Heller}
Let $(l,(\{M_i\},V),r)$ be a reduced stochastic module.
The associated stochastic process is induced from a
Markov chain if and only if there is a cone $\mathcal C$
contained in the vector space $V$  such that the following hold:
\begin{enumerate}
\item
 $r\in \mathcal C$,
\item
 $l\mathcal C \subset [0,\infty )$,
\item
$(A_S)_+ \mathcal C \subset \mathcal C$,
\item
 $\mathcal C$ is strongly convex and  polyhedral.
\end{enumerate}
\end{thm}
Heller stated this result in \cite[Theorem 1]{Heller1965}. The proof
there contained a minor error which was corrected in
\cite{Heller1967}. Heller defined a process to be {\em finitary} if
its associated reduced stochastic module is finite dimensional. (We
will call the corresponding measure finitary.) A consequence of
Theorem \ref{thm_Heller} is the (obvious) fact that the reduced
stochastic module of a sofic measure must be finitary. Heller gave
an example \cite{Heller1965} of a finitary process which is not a
1-block factor of a 1-step Markov measure, and therefore is not a
factor of any Markov measure. {}{ (However, a subshift with a weakly
mixing finitary measure is measure theoretically isomorphic to a
Bernoulli shift \cite{BinkowskaKaminski1984}.) }

\subsection{Linear automata and the reduced stochastic
module for a finitary measure}
\label{redstochsofic}

The 1960's and 1970's saw the development of the theory of probabilistic
automata  and
linear automata. We have not thoroughly reviewed this literature, and
 we may be missing from it significant points of contact
with and independent invention of the ideas under review.
However, we mention at least one.
A  finite dimensional stochastic module is a special case of
a linear  space
automaton, as developed
in  \cite{InagakiFutumuraMutuura1972} by
Inagaki, Fukumura and
Matuura, following earlier work on probabilistic
automata (e.g. \cite{Paz1971,Rabin1963}.
They associated to each
linear  space
automaton  its canonical (up to isomorphism)
equivalent irreducible linear space automaton.
When the linear space automaton is a stochastic module,
its irreducible linear space automaton
 corresponds exactly to Heller's canonical
(up to isomorphism) reduced stochastic module. Following
\cite{InagakiFutumuraMutuura1972} and Nasu's paper \cite{Nasu1985},
we will give some concrete results on the reduced stochastic module.

 We continue the Example \ref{stochsofic} and produce a
concrete version of the reduced stochastic module in the case that a
measure on a subshift is presented by a stochastic module which is
finite dimensional as a real vector space (for example, in the case
of a sofic measure). Our presentation follows a construction of Nasu
\cite{Nasu1985} (another is in \cite{InagakiFutumuraMutuura1972}).
Correspondingly, in this section we will reverse Heller's roles for
row and column vectors and regard the stochastic module as generated
by row vectors.

So, let $(u,\{M_i\},v)$ be a finite dimensional stochastic module on
finite alphabet $\mathcal A$. We take the presentation so that there
is a positive integer $n$ such that the $M_i$ are $n\times n$
matrices; $u$ and $ v$ are $n$-dimensional row and column vectors;
and  the map $a\mapsto M_a$ induces a monoid homomorphism $\phi$
from $\mathcal A^*$, sending a word $w=a_1\cdots a_j $ to the matrix
$\phi (w)= M_{a_1}\cdots M_{a_j} $.

Let $\mathcal U$ be the vector space generated by vectors of the
form $u\phi (w)$, $w\in \mathcal A^*$. Similarly define $\mathcal V$
as the vector space generated by vectors of the form $\phi (w)v$,
$w\in \mathcal A^*$. Let $k=\text{dim}(\mathcal U)$. If $k<n$, then
construct a smaller module (presenting the same measure) as follows.
Let $L$ be a $k\times n$ matrix whose rows form a basis of $\mathcal
U$. For each symbol $a$ there exists a  $k\times k$ matrix $\widehat
M_a$ such that $L M_a= \widehat M_a L $. Define $\widehat u$ to be
the $k$ dimensional row vector such that $\widehat uL =u$ and set
$\widehat v=Lv$. Let $a\to \widehat M_a$ induce a monoid homomorphism
$\widehat \phi$ from $\mathcal A^*$, sending a word $w=a_1\cdots a_j
$ to  $\widehat \phi (w)= \widehat M_{a_1}\cdots \widehat M_{a_j} $.
The subspace $\widehat{\mathcal U}$ of $\mathbb R^k$ generated by
vectors of the form $\widehat u \widehat \phi (w) $ is equal to
$\mathbb R^k$ because $\widehat{\mathcal U}L=\mathcal U$ and
$\text{dim}(\mathcal U)=k$.  It is easily checked that $\widehat u
\widehat \phi (w) \widehat v = u \phi (w) v$, for every $w$ in
$\mathcal A^*$.
 Let $\widehat{\mathcal V} $ be the subspace of $\mathbb R^k$
generated by column vectors $\widehat{\phi}(w)\widehat v$. We have
for each $a$ that $L M_av= \widehat M_a Lv= \widehat M_a \widehat v
$, so $L$ maps $\mathcal V$ onto $\widehat{\mathcal V}$. Also $L$
maps the space of $n$-dimensional column vectors onto $\mathbb R^k$.
It follows that if $\text{dim}(\mathcal V)=n$, then
 $\text{dim}(\widehat{\mathcal V})=k$.

If  $\text{dim}(\widehat{\mathcal V})<k$, then repeat the reduction
move, but applying it to $v$ (column vectors) rather than to $u$.
This will give
 a stochastic module
$(\overline u, \{\overline M_a\},\overline v )$, say with $m\times
m$ matrices $M_a$ and invariant subspaces $\overline{ \mathcal
U},\overline{\mathcal V}$ generated by the action on $\overline u,
\overline v$. By construction we have $\dim (\overline{\mathcal
V})=m$. And because $\widehat{\mathcal U}$ had full dimension, we
have $\dim (\overline{\mathcal U})=m$ also. {}{ Regarding $\mathcal
V$ as a space of functionals on $\mathcal U$, and letting
$\text{ker}(\mathcal V)$ denote the subspace of $\mathcal U$
annihilated by all elements of $\mathcal V$, we see that $u\mapsto
\overline u$ is a presentation of the map $\pi: \mathcal U\to
\mathcal U/\text{ker}(\mathcal V)$. Thus $(\overline u, \{\overline
M_a\},\overline v )$ is a presentation of the reduced stochastic
module. Also, for all $a$, $\pi M_a = \overline M_a \pi$, and
therefore the surjection $\pi$ (acting from the right) also
satisfies
\begin{equation}
 \label{intertwine}
\Big(\sum_a M_a \Big)\pi
= \pi \Big(\sum_a \overline M_a \Big)  \ .
\end{equation} }
If $(\widetilde u, \{\widetilde M_a\},\widetilde v )$ is another
such presentation of the reduced stochastic module, then it must
have the same (minimal) dimension $m$, and there will be  an
invertible matrix $G$ {}{(giving the isomorphism of the two
presentations)} such that for all
 $a$,
 \be \Big(\widetilde u, \ \{\widetilde M_a\},\ \widetilde v
\Big) = \Big(\overline uG, \ \{G^{-1}\overline M_aG\},\
G^{-1}\overline v \Big) \ .
 \en
 To find $G$, simply take $m$ words
$w$ such that the vectors $\overline{u}\overline{\phi} (w)$
{are a basis for}  $\overline \cu$, and let $G$ be the matrix such
that for each of these $w$, \be \overline u\overline{\phi} G =
\widetilde u\widetilde{\phi} \ . \en

The rows of the matrix $L$
above  (a basis for the space $\mathcal U$) may be obtained by examining
vectors $u\phi (w)$ in some order, with the length of $w$
nondecreasing, and including as a row any vector not in the span of
previous vectors. Let $\mathcal U_m$ denote the space spanned by
vectors $u\phi (w)$ with $w$ of length at most $m$.
 If for some $m$ it holds that $\mathcal U_m =
\mathcal U_{m+1}$, then $\mathcal U_m = \mathcal U$.
In particular,
if $n$ is the dimension of the original stochastic module,
then the matrix $L$ can be found by considering words of
length at most  $n-1$.

One can  check that
if  two equivalent stochastic modules have dimensions $n_1$ and $n_2$,
then they are equivalent (define the same measure) if and only if
they assign the same measure to words of length $n_1 + n_2 -1$.
(This is a special
case of \cite[Theorem 5.2]{InagakiFutumuraMutuura1972}.)
If the reduced stochastic module of a measure has dimension at
most $n$, then one can also construct the reduced stochastic
module from the measures of words of length at most $2n-1$
(one construction is given in
\cite[Theorem 6.2]{InagakiFutumuraMutuura1972}). However, without
additional information about the measure, this forces
the examination of a number of words which for a fixed
alphabet can grow exponentially
as a function of $n$, as indicated by the following example.
\begin{ex}
Let $X$ be the full shift on the three symbols $0,1,2$. Given $k\in
\mathbb N$, define a stochastic matrix $P$ indexed by $X$-words of
length $k+1$ by $P(10^k,0^k1)=1/6=P(20^k,0^k2)$;
$P(10^k,0^k2)=1/2=P(20^k,0^k1)$; $P(a_0\cdots a_{k},  a_1\cdots
a_{k+1}) =1/3$ otherwise; and all other entries  of $P$ are zero.
This matrix defines a $(k+1)$-step Markov measure $\mu$ on $X$ which
agrees with the Bernoulli $(1/3,1/3,1/3)$ measure on all words of
length at most $k+2$ except  the four words
$10^k1,10^k2,20^k1,10^k2$. The reduced stochastic module has
dimension at most $2k+1$, because for any word $U$ the conditional
probabilty function on $X$-words defined by $\rho_U: W\mapsto \mu
(UW|U)$ will be a constant multiple of $\rho_V$ for one of  the
words  $V= 0^{k+1},10^j, 20^j$, with $0\leq j \leq k$. The number of
$X$-words of length $k+2$ is  $3^{k+2}$.
\end{ex}

\subsection{Topological factors of finitary measures, and Nasu's
core matrix}
\label{remark_nasucore}

The content of this section is essentially taken from Nasu's paper
\cite{Nasu1985}, as we explain in more detail below. Given a square
matrix $M$, in this section we let $M^*$ denote any square matrix
similar to one giving the action of $M$ on the maximal invariant
subspace on which the action of $M$ is nonsingular.

Adapting terminology from \cite{Nasu1985}, we define the {\it core
matrix} of a finite dimensional stochastic module give by matrices,
$(l,\{M_i\},r)$, to be $\sum_i M_i$. A core matrix for a finitary
measure $\mu$ is any matrix which is the  core matrix of a reduced
stochastic module for $\mu$. This matrix is well defined only up to
similarity, but for simplicity of language we refer to {\it the core
matrix of} $\mu$, denoted Core($\mu$). Similarly, we define {\it the
eventual core matrix of} $\mu$ to be $\text{Core(}\mu\text{)}^*$,
denoted  $\text{Core}^*\text{(}\mu\text{)}$. E.g., if Core($\mu$) is
$\begin{pmatrix}
\frac 12 & 0 & 0 & 0 \\
0 & 1 & 0& 0 \\
0&0&0&1 \\
0&0&0&0
\end{pmatrix}
$,
then
$\text{Core}^*\text{(}\mu\text{)}$
is
$\begin{pmatrix}
\frac 12 & 0  \\
0 & 1
\end{pmatrix}
$.

Considering  square matrices $M$ and $N$ as linear
endomorphisms, we say $N$ is a quotient of $M$ if
there is a linear surjection $\pi$ such that,
writing action from the right, $M \pi = \pi N$.
(Equivalently, by duality, the action of $N$ is isomorphic
to the action of $M$ on some invariant subspace.)
 In this case, the characteristic polynomial of $M$
divides that of $N$ (but,  e.g.
$\begin{pmatrix}
2&0  \\
0&2
\end{pmatrix}
$
 is a principal submatrix of but not a quotient of
$\begin{pmatrix}
2&1&0  \\
0&2&1 \\
0&0&2
\end{pmatrix}
$).

\begin{thm} \label{coretheorem}
Suppose $\phi$ is a continuous factor map from a subshift $X$ onto a
subshift $Y$, $\mu \in \mathcal M(X)$ and $\phi \mu=\nu \in \mathcal
M(Y)$. Suppose $\mu$ is finitary.
 Then $\nu $ is finitary, and
$\text{Core}^*\text{(}\nu\text{)}$ is a quotient of
$\text{Core}^*\text{(}\mu\text{)}$. In particular, if $\phi$ is a
topological conjugacy, then $\text{Core}^*\text{(}\nu\text{)} =
\text{Core}^*\text{(}\mu\text{)} $.
\end{thm}

The key to the topological invariance in Theorem \ref{coretheorem}
is the following lemma (a measure version of
\cite[Lemma 5.2]{Nasu1985}).

\begin{lem}\label{corelemma}
Suppose $\mu$ is a finitary measure on a subshift $X$ and $n\in
\mathbb N$. Let $X^{[n]}$ be the $n$-block presentation of $X$; let
$\psi : X^{[n]}\to X$ the 1-block factor map defined on symbols by
$[a_1\cdots a_n] \mapsto a_1$; let $\mu^{[n]}\in \mathcal
M(X^{[n]})$ be the measure such that $\psi\mu^{[n]}= \mu$.
 Then
$\mu^{[n]}$ is finitary and $\text{Core}^*\text{(}\mu^{[n]}\text{)}$
is a quotient of $\text{Core}^*\text{(}\mu\text{)}$.
\end{lem}

\begin{proof}[Proof of Lemma \ref{corelemma}]
For $n>1$, the $n$-block presentation of $X$ is (after a
renaming of the alphabet) equal to the 2-block presentation
of $X^{[n-1]}$. So, by induction it suffices to
prove the lemma for $n=2$.

Let $(l, \{P_i\},r)$ be a reduced stochastic
module for $\mu$, where the $P_i$ are $k\times k$
and $\mathcal A(X)=\{1,2,\dots ,m\}$.
For each symbol $ij$ of $\mathcal A(X^{[2]})$, define an
$mk\times mk$ matrix $P'_{ij}$ as an $m\times m$
system of $k\times k$ blocks, in which the
$i,j$ block is $P_i$ and the other entries
are zero. Define $l'=(l,\dots ,l)$
($m$ copies of $l$) and define
$r' =
\begin{pmatrix}
P_1r \\
\vdots \\
P_mr
\end{pmatrix}
$.
Then $(l',\{P'_{ij}\},r')$ is a stochastic module for
$\mu^{[2]}$, which is therefore finitary. Also, we have
an elementary strong shift equivalence of the core
matrices $P$ and $P'$,
\[
P' =
\begin{pmatrix}
P_1 \\
\vdots \\
P_m
\end{pmatrix}
\begin{pmatrix}
I &
\cdots &
I
\end{pmatrix}
\qquad , \qquad
P=
\begin{pmatrix}
I &
\cdots &
I
\end{pmatrix}
\begin{pmatrix}
P_1 \\
\vdots \\
P_m
\end{pmatrix} ,
\]
and therefore $P^*=(P')^*$.
Because
$\text{Core}\text{(}\mu^{[2]}\text{)}$
 is a quotient of $P'$,
it follows that
$\text{Core}^*\text{(}\mu^{[2]}\text{)}$ is  a
quotient of $(P')^*=P^*=
\text{Core}^*\text{(}\mu\text{)}$.
\end{proof}

If $\phi : X\to Y$ is a factor map of irreducible sofic
shifts of equal entropy, then $\phi$ must send the
unique measure of maximal entropy of $X$, $\mu_X$,
 to that for
$Y$. These are sofic measures, and consequently
Theorem \ref{coretheorem} gives computable obstructions
to the existence of such a factor map between given
$X$ and $Y$.
In his work, Nasu associated to given $X$ a certain
linear (not stochastic)
automaton. If we denote it $(l,\{M_i\},r)$,
and let $\log (\lambda)$ denote the topological entropy of
$X$, then $(l,\{(1/\lambda )M_i\},r)$ would be a stochastic
module for $\mu_X$. In the end Nasu's core matrix is
$\lambda \text{Core}\text{(}\mu_X\text{)}$.
Nasu remarked in \cite{Nasu1985}
that his arguments could as well be
carried out with respect to measures to obtain his results,
and that is what we have done here.

Eigenvalue relations between core matrices (not so named)
of equivalent linear automata already appear in
\cite[Sec.7]{InagakiFutumuraMutuura1972}. Also,
Kitchens
\cite{Kitchens1981} earlier used the (Markov) measure of
maximal entropy for an irreducible shift of finite type
in a similar way to show that the existence of a factor
map of equal-entropy irreducible SFTs, $\Omega_A \to
\Omega_B$, implies (in our terminology)
that $B^*$ is a quotient of $A^*$. This is a special
case of Nasu's constraint.

%

\section{When is a sofic measure Markov?}
\label{sec_markov}

\subsection{When is the image of a 1-step Markov measure under a 1-block
map 1-step Markov?}\label{sec_when}

We return to considering Problem \ref{prob_1}. In this subsection,
suppose $\mu$ is a 1-step Markov measure, that is, a 1-step fully
supported shift-invariant Markov measure on an irreducible shift of
finite type $\Omega_A$. Suppose that $\pi$ is a 1-block code with
domain $\Omega_A$. How does one characterize the case
when the measure
 $\pi\mu$
 is again 1-step Markov?

{To our knowledge, this problem was introduced, in the language of
Markov processes, by Burke and Rosenblatt (1958)
 \cite{BurkeRosenblatt1958}, who solved it in
the reversible case \cite[Theorem 1]{BurkeRosenblatt1958}.} Kemeny
and Snell \cite[Theorems 6.4.8 and 6.3.2]{KemenySnell1960} gave
another exposition
and introduced the ``lumpability'' terminology.
 Kemeny and Snell defined a (not necessarily
stationary) finite-state Markov process $X$ to be {\it lumpable}
with respect to a partition of its states if for every initial
distribution for $X$ the corresponding quotient process is Markov.
They defined $X$ to be {\it weakly lumpable} with respect to the
partition if there exists an initial distribution for $X$ for which
the quotient process $Y$ is Markov. In all of this, by Markov they
mean 1-step Markov.
Various problems around
these ideas were (and continue to be) explored and solved.
For now
 we restrict our attention to the question of the title of this
subsection and describe three answers.

\subsubsection{Stochastic module answer}

\begin{thm}\label{thm_markovrank}  Let $(l,M,r)$ be
a presentation of the reduced stochastic module of a sofic measure
$\nu$ on $Y$, in which $M_i$ denotes the matrix by which a symbol
$i$ of $\mathcal A(Y)$ acts on the module. Suppose $k\in \mathbb N$.
Then the sofic measure $\nu$ is $k$-step Markov if and only if every
product $M_{i(1)}\cdots M_{i(k)}$ of length $k$ has rank at most 1.
\end{thm}

The case $k=1$ of Theorem \ref{thm_markovrank} was proved by Heller
\cite[Prop.3.2]{Heller1965} An equivalent characterization was given
a good deal later, evidently without awareness of Heller's work, by
Bosch \cite{Bosch1974/75}, who worked from the papers of Gilbert
\cite{Gilbert1959} and Dharmadhikari \cite{Dharma1963}. The case of
general $k$ in Theorem \ref{thm_markovrank} was proved by Holland
\cite[Theorem 4]{Holland1968}, following Heller.

\subsubsection{Linear algebra answer}

{ One can approach the problem of deciding whether a sofic measure
is Markov with straight linear algebra. There is a large literature
using such ideas in the context of automata, control theory and the
``lumpability'' strand of literature emanating from Kemeny and Snell
(see e.g. \cite{GurvitsLedoux2005} and its references). Propositions
\ref{linalgprocedure} and \ref{kmarkov} and Theorem
\ref{theoremsecondmarkov} are taken from Gurvits and Ledoux
\cite{GurvitsLedoux2005}. As with previous references, we are
considering only a fragment of this one. }

Let $N$ be the size of the alphabet of the irreducible shift of
finite type $\Omega_A$. Let $\pi$ be a 1-block code mapping
$\Omega_A$ onto a subshift $Y$. Let $P$ be an $N\times N$
irreducible stochastic matrix defining a 1-step Markov measure $\mu$
on $\Omega_A$. Let $p$ be the positive stochastic row
fixed  vector of $P$. Let $U$ be the matrix such
that $U(i,j)=1$ if $\pi$ maps the state $i$ to the state $j$, and
$U(i,j)=0$ otherwise. Given $i\in \mathcal A(\Omega_A)$, let
$\overline i$ be its image symbol in $Y$. Given  $j\in \mathcal
A(Y)$, let $P_j$ be the matrix of size $P$ which equals $P$  in
columns $i$ such that $\overline i = j$, and is zero in other
entries. Likewise define $p_j$.
{Given a $Y$-word $w=j_1\cdots j_k$, we let $P_w=P_{j_1}\cdots
P_{j_k}$.}

{ Alert: We are using parenthetical notation for matrix and vector
entries and subscripts for lists.}
 If $\pi
\mu$ is a 1-step Markov measure on $Y$, then it is defined
using a stochastic row vector $q$ and stochastic matrix $Q$. The
vector $q$ can only be $pU$, and the entries of $Q$ are determined
by $Q(j,k) =(p_jP_{k}U)(k)/q(j)$.
 Let $\nu$ denote the Markov
measure defined using $q,Q$. Define $q_j, Q_j$ by replacing entries
of $q,Q$ with zero in columns not indexed by $j$. For a word
$w=j_0\dots j_k$ on symbols from $\mathcal A(Y)$, we have $ (\pi
\mu) (\mathcal C_0(w)) = \nu (\mathcal C_0(w)) $ if and only if \be
\label{markovcondition} p_{j_0}P_{j_1}\cdots P_{j_k}U = p_{j_0}U
Q_{j_1}\cdots Q_{j_k} \ \ \en (since $q_{j_0}=p_{j_0}U $). Thus
$\pi \mu = \nu$ if and only if (\ref{markovcondition})
 holds for all $Y$-words
 $w$.
This remark is already more or less in
Kemeny and Snell \cite[Theorem
6.4.1]{KemenySnell1960}.

For the additional argument which produces a finite procedure, we
define certain vector spaces (an idea already in
 \cite{Ellis,Kitchens1982,RubinoSericola1989,RubinoSericola1991,
GurvitsLedoux2005} and elsewhere).

Let $\mathcal V_k$ denote the real vector space
generated by the row vectors $p_{j_0}P_{j_1}\cdots P_{j_k} $
 such
that $j_0j_1\cdots j_t$ is a $Y$-word and $0\leq t \leq k$. So,
$\mathcal V_0$ is the vector space generated by the vectors
$p_{j_0}$, and $\mathcal V_{k+1}$ is the subspace generated by
$\mathcal V_k \cup \{vP_j: v\in \mathcal V_k, j\in \mathcal A(Y)\}$.
{}{ In fact, for $k\geq 0$, we claim that
\begin{align}
\label{vfirst}
\mathcal V_k &=
\langle \
\{
p_{j_0}P_{j_1}\cdots P_{j_k}:
j_0\cdots j_k \in \mathcal A(Y)^{k+1}
\}\
\rangle
 \ , \ \textnormal{ and} \\
\label{vsecond}
\mathcal V_{k+1} &=
\langle \
\{
vP_j:
v\in \mathcal V_k, j\in \mathcal A(Y)
\}\
\rangle \ ,
\end{align}
where $\langle \ \ \rangle$ is used to denote span.
Clearly (\ref{vsecond}) follows from (\ref{vfirst}), which
is a consequence of stationarity, as follows.
Because $\sum_jp_j = p = pP =\sum_jpP_j$, and for $i\neq j$
the vectors $p_i$ and $pP_j$ cannot both be nonzero
in any coordinate, we $p_j=pP_j$. So, given $t$ and
$j_1\cdots j_t$, we have
\begin{align*}
p_{j_1}P_{j_2}\cdots P_{j_t} &=
pP_{j_1}P_{j_2}\cdots P_{j_t} \\
&= \sum_{j_0} p_{j_0}P_{j_1}P_{j_2}\cdots P_{j_t} ,
\end{align*}
from which (\ref{vsecond}) easily follows. Let $\mathcal V=
\langle \cup_{k\geq 0}\mathcal V_k\rangle $.
}

\begin{prop} \label{linalgprocedure}
Suppose $P$ is an $N\times N$ irreducible stochastic matrix and
$\phi$ is a 1-block code.
 Let the vector spaces $\cv_k$ be defined as above, and let
 $n$ be the smallest positive integer such that
$\mathcal V_n=\mathcal V_{n+1}$. Then $n\leq  N -|\mathcal A (Y)|$,
$\mathcal V_n=\mathcal V$, and the following are
equivalent:
\begin{enumerate}
\item
$\phi \mu $ is a 1-step Markov measure on the image  of $\phi$.
\item
$p_{j_0}P_{j_1}\cdots P_{j_n}U = p_{j_0}U
Q_{j_1}\cdots Q_{j_n}$,
for all
$j_0\cdots j_n \in
\mathcal A (Y)^{n+1}$.
\end{enumerate}
\end{prop}

\begin{proof}
For $k\geq 1$, we have $\mathcal V_k \subset \mathcal V_{k+1}$,
and also
\be
\mathcal V_k = \mathcal V_{k+1} \quad\text{implies}\quad \mathcal
V_k = \mathcal V_{t} =\mathcal V
\ \ \ \textnormal{for all }t\geq k \ .
\en
Because $\dim (\mathcal V_0) = |\mathcal A(Y)|$,
it follows  that $n\leq   N - |\mathcal A(Y)|$.

Because (1) is equivalent to
(\ref{markovcondition})  holding for all $Y$-words
$j_0j_1\cdots j_k, \ k \geq 0$, we have that (1) implies (2).

Now suppose (2) holds. For $K\geq 1$, the linear condition
(\ref{markovcondition})  holds for all $Y$-words of length $k$ less
than or equal to $K$ if and only if $vUQ_j = vP_jU$ for all $j$ in
$\mathcal A(Y)$ and all $v$ in $\mathcal V_K$.
($U$ is the matrix defined above.)
 Because $\cv_K=\cv_n$ for $K\geq n$,
  we conclude
from (2) {}{and (\ref{vfirst})}
 that (\ref{markovcondition}) holds for all
 $Y$-words
$j(0)j(1)\cdots j(k),\ k \geq 0$, and therefore (1) holds.
\end{proof}

{}{Next we consider an irrreducible $N\times N$ matrix $P$ defining
a 1-step Markov measure $\mu$ on $\Omega_A$ and a 1-block code
$\phi$ from $\Omega_A$ onto a subshift $Y$. Given a positive integer
$k\geq 1$, we are interested in understanding when $\phi \mu$ is
$k$-step Markov. We use notations $U,p,p_j,P_j, \mathcal V_t $ and
$\mathcal V_n = \mathcal V$ as above. Define a stochastic row vector
$q$ indexed by $Y$-words of length $k$, with $q(j_0\cdots j_{k-1}) =
(p_{j_0} P_{j_1}\cdots P_{j_{k-1}}U)(j_{k-1}). $ Let $Q$ be the
square  matrix indexed by $Y$-words of length $k$ whose nonzero
entries are defined by
\[
Q(j_0\cdots j_{k-1},j_1\dots j_k)
= \frac
{\Big( p_{j_0}P_{j_1}\cdots P_{j_k}U\Big) (j_k)}
{q(j_0\cdots j_{k-1})} \ .
\]
Then $Q$ is an irreducible stochastic matrix and $q$ is a positive
stochastic vector such that $qQ=q$. Let $\nu$ be the $k$-step Markov
measure defined on $Y$ by $(q,Q)$. The measures $\nu$ and $\phi \mu$
agree on cylinders $\cc_0 (j_0\cdots j_k)$ and therefore on all
cylinders $\cc_0 (j_0 \cdots j_t)$ with $0\leq t \leq k$. Clearly,
if $\phi \mu$ is $k$-step Markov then $\phi \mu$ must equal $\nu$. }

{}{
\begin{prop} \label{kmarkov} \cite{GurvitsLedoux2005}
Suppose $P$ is an $N\times N$ irreducible stochastic matrix defining
a 1-step Markov measure $\mu$
on $\Omega_A$ and $\phi: \Omega_A \to Y $ is
a 1-block code. Let $k$ be a fixed positive integer.
With the notations above, the following are equivalent.
\begin{enumerate}
\item
$\phi \mu $ is a $k$-step Markov measure
(i.e., $\phi \mu = \nu$).
\item For every $Y$-word $w=w_0\cdots w_{k-1}$
of length $k$ and
every $v\in \mathcal V$,
 \be \label{firstkmarkovcondition}
vP_w(PU-\mathbf 1Q^{w}) =
 0 ,
 \en
 where $P_w=P_{w_0}\cdots P_{w_{k-1}}$;
$\mathbf 1$ is the size $N$ column vector with every
entry 1;  and
 $Q^{w}$ is the stochastic row vector defined by
 \be Q^{w}(j) = Q(w_0\cdots w_{k-1},w_1\cdots w_{k-1}j) \  , \quad
 j\in \mathcal A(Y)\ .
 \en
\end{enumerate}
\end{prop}
}

 \begin{proof}
 We continue to denote by
$z(j)$ the
entry in the $j$'th coordinate of a row
 vector
$z$. By  construction of $\nu$ we have for $t=0$ that \be
\label{inductionstatement} (\pi \mu)\cc_0(j_0\cdots j_{t+k}) = \nu
\cc_0(j_0\cdots j_{t+k})\  \quad \textnormal{for all }j_0\cdots
j_{t+k}\in \mathcal A^{t+k+1} \ . \en Now suppose $t$ is a
nonnegative integer and (\ref{inductionstatement}) holds for $t$.
Given
 $j_0\cdots j_{t+k}$, let $w$ be its terminal word of
length $k$. Then for  $j\in \mathcal A(Y)$,
\begin{align*}
&(\pi \mu)\cc_0(j_0\cdots
j_{t+k}j) -
\nu \cc_0(j_0\cdots
j_{t+k}j)
\\
 = \
&\Big( p_{j_0}P_{j_1}\cdots P_{j_{t+k}}P_jU\Big)(j) -
\Big(\nu \mathcal C_0(j_0\cdots j_{t+k})Q^{w}\Big)(j) \\
=\ &\Big( p_{j_0}P_{j_1}\cdots P_{j_{t+k}}P_jU\Big)(j) -
\Big( (p_{j_0}P_{j_1}\cdots
P_{j_{t+k}}\mathbf 1)Q^{w}\Big) (j) \\
=\ &\Big( p_{j_0}P_{j_1}\cdots P_{j_{t+k}}
[P_jU -\mathbf 1 Q^{w}]\Big)(j)  \\
=\ &\Big( p_{j_0}P_{j_1}\cdots P_{j_{t}}P_w [PU -\mathbf 1
Q^{w}]\Big)(j) ,
\end{align*}
where  the term $P_{j_1}\cdots P_{j_{t}}$ is included only if $t>0$,
and the last equality holds because
the $j$th columns of $PU$ and $P_jU$ are equal.
Thus, given (\ref{inductionstatement}) for $t$,
by (\ref{vfirst}) we
have (\ref{inductionstatement}) for $t+1$ if and only
$vP_w[PU -\mathbf 1 Q^{w}]=0$ for all $v\in \mathcal V_{t}$
and all $w$ of length $k$.
It follows from induction that (\ref{inductionstatement})
holds for all $t\geq 0$ (i.e. $\pi \mu = \nu$) if and only
if
(\ref{firstkmarkovcondition}) holds for all $v\in \mathcal V$.
\end{proof}

Because $\mathcal V$ can be computed,  Proposition
\ref{kmarkov} gives an algorithm, given $k$,  for determining whether
the image of a 1-step Markov measure is a $k$-step Markov measure.
The next result gives a criterion which does not require computation
of the matrix $Q$.

 {
\begin{thm} \cite{GurvitsLedoux2005} \label{theoremsecondmarkov}
Let notations be as in Proposition
\ref{kmarkov}. Then $\phi \mu $ is a $k$-step Markov measure on $Y$
if and only if for every $Y$-word $w$ of length $k$, \be
\label{secondmarkovcondition} \Big((\mathcal VP_w)\cap \ker(U)\Big)P
\subset \ker(U) \ . \en
\end{thm}
}
{
\begin{proof}
Let $w=w_0\cdots w_{k-1}$ be a $Y$-word of length $k$. Using the
computations of the proof of Proposition \ref{kmarkov}, we obtain for
$j\in \mathcal A(Y)$ that
\begin{align*}
0 \ =&\
\pi \mu \mathcal C_0(w_0\cdots w_{k-1}j)  -
\nu \mathcal C_0(w_0\cdots w_{k-1}j)  \\
=&\
\Big( p_{w_0}P_{w_1}\cdots P_{w_{k-1}} [PU-\mathbf 1Q^w] \Big) (j) \\
 =&\
\Big( pP_{w_0}P_{w_1}\cdots P_{w_{k-1}} [PU-\mathbf1Q^w] \Big) (j) \\
=&\ \Big( pP_w[PU-\mathbf 1Q^w] \Big) (j) \ .
\end{align*}
Consequently, the vector $v=p$ satisfies
(\ref{firstkmarkovcondition}). Moreover,
\[
(pP_wU)(w_{k-1}) =(p_{w_0}P_{w_1}\cdots P_{w_{k-1}}U)(w_{k-1}) =
\pi\mu\mathcal C_0(w)>0 ,
\]
and therefore $pP_w\notin \text{ker}(U)$.
Because $vP_w=0$ if and only if $vP_w\mathbf 1=0$, the space
$\mathcal VP_w$ is spanned by $pP_w$ and $(\mathcal VP_w)\cap
\text{ker}(U)$. Thus (\ref{firstkmarkovcondition}) holds for all
$v\in \mathcal V$ if and only if
(\ref{firstkmarkovcondition}) holds for all $v\in \mathcal V$ such
that
$vP_w\in \text{ker}(U)$, which is equivalent to
(\ref{secondmarkovcondition}).
\end{proof}
}

{ Gurvits and Ledoux \cite[Sec. 2.2.2]{GurvitsLedoux2005} explain
how Theorem \ref{theoremsecondmarkov} can be used to produce an
algorithm, polynomial in the number $N$ of states,
 for deciding
whether  $\pi \mu$ is a 1-step Markov measure.
 }

\subsection{Orders of Markov measures under codes}\label{sec_orders}

This section includes items relevant to
the second part of Problem \ref{prob_1}.

\begin{defn} Given positive integers $m,n,k$ with $1\leq k \leq n$,
recursively define integers $N(k,m,n)$ by setting
\begin{align}
N(n,m,n) &= 1  \\
N(k,m,n) &= (1+m^{N(k+1,m,n)})N(k+1,m,n)\ ,
\quad \text{if } 1\leq k < n \ .
\end{align}
\end{defn}

\begin{prop}\label{prop_stepbound}
Suppose $\pi:\Omega_A \to Y$ is a 1-block code and $\mu$ is a 1-step
Markov measure on $\Omega_A$. Let $n$ be the dimension of the
reduced stochastic module of $\pi \mu$ and let $m=|\mathcal A(Y)|$.
Suppose $n\geq 2$.
  (In the case $n=1$, $\pi \mu$ is Bernoulli.) Let
$K=N(2,m,n)$. If $\pi \mu$
 is not $K$-step Markov, then it is not $k$-step Markov for
any $k$.
\end{prop}

Before proving Proposition \ref{prop_stepbound},
 we state our main interest in it.

\begin{cor} \label{cor_markovalgorithm}
Suppose $\mu$ is a 1-step Markov measure on an irreducible SFT
$\Omega_A$ determined by a stochastic matrix $P$, and that there are
algorithms for doing arithmetic in the field generated by the
entries of $P$. Suppose  $\phi$ is a block code on $\Omega_A$. Then
there is an algorithm for deciding whether the measure $\phi \mu$ is
Markov.
\end{cor}
\begin{proof}
The corollary is an easy consequence of
Propositions  \ref{linalgprocedure} and
\ref{prop_stepbound}.
\end{proof}

The proof of Proposition \ref{prop_stepbound} uses two lemmas.

\begin{lem}\label{lem_prodrank}
    Suppose $P_1, \dots ,P_t$ are $n \times n$ matrices such that
 $\textnormal{rank}
(P_1 ... P_tP_1)= \textnormal{rank} (P_1) = r$. Then for all
positive
integers $m$,
 $\textnormal{rank}
(P_1 ... P_t)^mP_1= r$.
\end{lem}
\begin{proof}
It follows from  the rank equality that
$(P_1...P_k)$ defines an isomorphism
from the image of $P_1$ (a vector space of column vectors)
to itself.
\end{proof}

\begin{lem}\label{lem_boundrank}
Suppose  $k,m,n$ are positive integers and $1\leq k \leq n$.
Suppose $\mathcal Q$ is a collection of $m$ matrices of size
$n \times n$,
 and there exists a product  of  $N(k,m,n)$  matrices
from $\mathcal Q$ with rank at least $k$.
Then there are arbitrarily long products of matrices from
$\mathcal Q$ with rank
at least $k$.
\end{lem}
\begin{proof}

We prove the proposition by induction on $k$, for $k$ decreasing
from $n$. The case $k=n$ is clear. Suppose now $1\leq k < n$ and
the lemma holds for $k+1$.
Suppose a matrix $M$ is
a  product  $Q_{i(1)}\cdots Q_{i(N(k,m,n))}$
of $N(k,m,n)$ matrices from $\mathcal Q$ and has rank
at least $k$. We must show there are arbitrarily long products
from $\mathcal Q$ with rank at least $k$.

The given product is a concatenation of products of length
 $N(k+1,m,n)$, and we define corresponding matrices,
\be
P_j = Q_{1+ (j-1)(N(k+1,m,n))}  \cdots Q_{j(N(k+1,m,n))}  ,
\quad 1\leq j \leq 1+m^{N(k+1,m,n)} \ .
\en
If any $P_j$ has
rank at least $k+1$, then by the induction hypothesis there
are arbitrarily long products with rank at least $k+1$, and
we are done. So, suppose every $P_j$
 has rank at most $k$. Because
$\text{rank}(P_j)\geq \text{rank}(M) \geq k$,
 it follows that $M$, and every $P_j$,
and every subproduct of consecutive $P_j$'s, has
 rank $k$.

There are only $m^{N(k+1,m,n)}$ words of length $N(k+1,m,n)$
on $m$ symbols, so two of the matrices $P_j$ must be equal.
The conclusion now follows from Lemma \ref{lem_prodrank}.
\end{proof}

 {\em Proof of Proposition \ref{prop_stepbound}}.
As described in Examples \ref{stochsofic} and \ref{redstochsofic},
there are algorithms for producing the reduced stochastic module for
$\pi \mu$ as a set of matrices $M_a$ (one for each symbol from
$\mathcal A(Y)$) and a pair of vectors $u,v$ such that for any
$Y$-word $a_1\cdots a_t$, $(\pi \mu) \mathcal C_0 (a_1\cdots a_t)=
uM_{a_1}\cdots M_{a_t}v$. By  Theorem \ref{thm_markovrank},
 $\pi \mu$ is $k$-step Markov if and only
every product $M_{a_1}\cdots M_{a_k}$ has rank at most 1.
Let $K=N(2,m,n)$.
If $\pi \mu$ is not $K$-step Markov,
then some matrix $\prod_{i=1}^KM_{a(i)} $
has rank at least 2, and by Lemma \ref{lem_boundrank}
there are then arbitrarily long products of $M_a$'s with
rank at least 2. By  Theorem \ref{thm_markovrank},
this shows that
$\pi \mu$ is not $k$-step Markov for any $k$.
 \qed

\begin{rem} Given $m$ and $n$, the numbers $N(k,m,n)$
grow very  rapidly as $k$ decreases.
Consequently, the
 bound $K$ in Proposition \ref{prop_stepbound}
 (and
consequently the algorithm of
Corollary \ref{cor_markovalgorithm}) is not practical.
However, in an analogous case (Problem
\ref{prob_markovianbound} below) we don't even know
the existence of an algorithm.
\end{rem}

\begin{prob}\label{prob_Kbound}
Find a reasonable bound $K$ for Proposition \ref{prop_stepbound}.
\end{prob}
{}{
\begin{ex} This is an example to show that the cardinality
of the domain alphabet cannot be used as the bound $K$ in
Proposition \ref{prop_stepbound}. Given $n>1$ in $\mathbb N$, let
$A$ be the adjacency matrix of the directed graph $\mathcal G$ which
is the union of two cycles, $a_1b_1b_2\cdots b_{n+4}a_1$ and
$a_2b_3b_4\cdots b_{n+3}a_2$. The vertex set $\{a_1,a_2,b_1,\dots
,b_{n+4}\}$ is the alphabet $\mathcal A$ of $\Omega_A$. Let $\phi$
be the 1-block code defined by erasing subscripts, and let $Y$ be
the subshift which is the image of $\phi$, with alphabet $\{a,b\}$.
Let $\mu$ be any 1-step Markov measure on $\Omega_A$. In $\mathcal
G$, there are exactly four first return paths from $\{a_1,a_2\}$ to
$\{a_1,a_2\}$: $a_1b_1\cdots b_{n+4}a_1$, $a_1b_1\cdots b_{n+3}a_2$,
$a_2b_3\cdots b_{n+4}a_1$ and $a_2b_3\cdots b_{n+3}a_2$. Thus, in a
point of $Y$, successive occurrences of the symbol $a$ must
correspondingly be separated by $m$ $b$'s, with $m\in \{ n+4, n+3,
n+2, n+1 \}$. Each $Y$-word $ab^ma$ has a unique preimage word, so
$\phi : \Omega_A \to Y$ is a topological conjugacy. Thus $\phi \mu$
is $k$-step Markov for some $k$. We have
\begin{align*}
\phi (b_1\cdots b_{n+3}a_2b_3\cdots b_{n+3}a_2)
&= \big( b^{n+3}ab^{n+1}\big) a \ , \quad \text{and} \\
 \phi (a_1b_1\cdots b_{n+4}a_1b_1\cdots b_{n+1})
&= ab\big( b^{n+3}ab^{n+1}\big)   \ .
\end{align*}
So, $\big( b^{n+3}ab^{n+1}\big)a$ and $ab\big(b^{n+3}ab^{n+1}\big)$
are $Y$-words, but $ab\big( b^{n+3}ab^{n+1}\big) a$ is not a
$Y$-word. Consequently, we have conditional probabilities,
\begin{align*}
\phi \mu [y_0 =a\ |\ y_{-(2n+5)}\cdots y_{-1}
&= \ \ \big(b^{n+3}ab^{n+1}\big)] > 0 \ , \\
\phi \mu [y_0 =a\ |\ y_{-(2n+7)}\cdots y_{-1} &=
ab\big(b^{n+3}ab^{n+1}\big)] = 0 \ ,
\end{align*}
which shows that $\phi \mu$ cannot be $(2n+5)$-Markov. In contrast,
$|\mathcal A| = n+6 < 2n +5 $.
\end{ex}
}

With regard to the problem (\ref{prob_markovianmap})
of determining whether a given factor map
is Markovian, the analogue of Proposition \ref{prop_stepbound}
is the following open problem.

\begin{prob}\label{prob_markovianbound} Find (or prove there does not exist)
 an algorithm for attaching to any
1-block code $\phi$ from an irreducible shift of finite type
a number $N$ with the following property:
 if a 1-step Markov measure $\mu$ on the range of $\phi$
has no preimage measure which is $N$-step Markov, then $\mu$ has no
preimage measure which is Markov.
\end{prob}

\begin{rema}\textnormal{
(The persistence of memory) Suppose  $\phi :\Omega_A \to \Omega_B $
is a 1-block code from one irreducible 1-step SFT onto another. We
collect some facts on how the memory of a Markov measure and a
Markov image must or can be related.
\begin{enumerate}
\item The image of a 1-step Markov measure can be Markov
but not $1$-step Markov. (E.g. the standard map from the $k$-block
presentation to the 1-block presentation takes the $1$-step Markov
measures onto the $k$-step Markov measures.)
\item
If $\phi$ is finite-to-one and $\nu$ is $k$-step Markov on
$\Omega_B$, then there is a unique Markov measure $\mu$ on
$\Omega_A$ such that $\phi \mu=\nu$, and $\mu$ is also $k$-step
Markov (Proposition \ref{prop_finitetoone}).
\item
If any $1$-step Markov measure on $\Omega_B$ lifts to a $k$-step
Markov measure on $\Omega_A$, then for every $n$, every $n$-step
Markov measure on $\Omega_B$ lifts to an $(n+k)$-step Markov measure
on $\Omega_A$. (This follows from the explicit construction
(\ref{markovmarkovian}) and passage as needed to a higher block
presentation.)
\item
If $\phi$ is infinite-to-one then it can happen \cite[Section
2]{BoyleTuncel1984} (``peculiar memory example'') that every 1-step
Markov measure on $\Omega_B$ lifts to a $2$-step Markov measure on
$\Omega_A$ but not to a 1-step Markov measure, while every 1-step
Markov on $\Omega_A$ maps to a $2$-step Markov measure on
$\Omega_B$.
\end{enumerate}
}
\end{rema}

\section{Resolving maps and Markovian maps}\label{sec_markandresolve}

In this section,
 $\Omega_A$ denotes an irreducible 1-step shift of
finite type defined by an irreducible matrix $A$.

\subsection{Resolving maps}\label{sec_resolving}

 In this section,
$\pi :\Omega_A \to Y$ is a 1-block code onto a subshift $Y$, with
$Y$ not necessarily a shift of finite type, unless specified. $U$
denotes the $0,1$, $|\mathcal A(\Omega_A )| \times |\mathcal A(Y )|$
matrix such that $U(i,j)=1$ iff $\pi (i)=j$. Denote a symbol $(\pi
x)_0$ by $\overline{x_0}$.
\begin{defn}
The factor map $\pi$ as above is {\em right resolving} if for all
symbols $i,\overline i, k$ such that $\overline i k $ occurs in $Y$,
there is at most one $j$ such that $ij$ occurs in $\Omega_A$ and
$\overline j=k$. In other words,  for any diagram
\be
\begin{CD}
i &  &  \\
@VVV & & \\
\overline{i}& @>>> & k
\end{CD}
\en
there is at most one $j$ such that
\be
\begin{CD}
i & @>>> & j \\
@VVV & & @VVV\\
\overline{i}& @>>> & k
\end{CD}
\en
\end{defn}
\begin{defn}
A factor map $\pi$ as above is {\em right e-resolving} if it
satisfies the definition above, with ``at most one'' replaced by
``at least one''.
\end{defn}
Reverse the roles of $i$ and $j$ above to define {\em left
resolving} and {\em left e-resolving}.
A map $\pi$ is {\em resolving
(e-resolving)} if it is left or right resolving (e-resolving).

\begin{prop}\label{resolvingprop}
\begin{enumerate}
\item
If $\pi$ is resolving, then $h(\Omega_A)=h(Y)$.
\item
If $Y=\Omega_B$ and $h(\Omega_A) =h(\Omega_B)$, then $\pi$ is
e-resolving iff $\pi$ is resolving.
\item
If $\pi$ is e-resolving, then $Y$ is a 1-step shift of finite type,
$\Omega_B$.
\item If $\pi$ is e-resolving and $k\in \mathbb N$,
then every $k$-step Markov measure on $Y=\Omega_B$ lifts to a
$k$-step Markov measure on $\Omega_A$.
\end{enumerate}
\end{prop}
\begin{proof}
(1) This holds because a resolving map must be finite-to-one
\cite{LindMarcus1995,Kitchens1998}.

(2) We argue as in \cite{LindMarcus1995,Kitchens1998}. Suppose $\pi$
is right-resolving. This means precisely
 that $AU\leq UB$. If $AU \neq UB$, then it would be possible
to increase some entry of $A$ by one and have a resolving
map onto $\Omega_B$ from some irreducible SFT
$\Omega_C$ properly containing $\Omega_A$. But now
$h(\Omega_C)>h(\Omega_A)$, while
$h(\Omega_C)=h(\Omega_B)=h(\Omega_A)$ because the resolving
maps respect entropy. This is a contradiction. The other direction
holds by a similar argument.

(3) This is an easy exercise \cite{BoyleTuncel1984}.

(4) We consider $k=1$ (the general case follows by passage to the
higher block presentation). Suppose $\pi$ is right e-resolving. This
means that $AU \geq UB$. Suppose $Q$ is a stochastic matrix defining
a 1-step Markov measure $\mu$ on $\Omega_B$. For each positive entry
$B(k,\ell )$ of $B$ and $i$ such that $\pi (i)=k$, let $\mathcal
J(i,k,l)$ be the set of indices $j$ such that $A(i,j)>0$ and $\pi
(j)=\ell$. Now simply choose $P$
to be
 any nonnegative matrix of size and zero/positive pattern
matching $A$ such that for each
 $i,k,l$, $\sum_{{j\in \mathcal J(i,k,l)}} P(i,j)
=Q(k,\ell )$.
 Then $PU=UQ$, and this guarantees that $\pi \mu = \nu$.
The condition on {the} +/0 pattern guarantees that $\mu$ has full
support on $\Omega_A$. (The code $\pi$ in Example \ref{exliftone} is
right e-resolving, and (\ref{resolvingsplit}) gives an example of this
construction.)
\end{proof}

The resolving maps, and the maps which are topologically equivalent
to them (the {\em closing} maps), form the  only class of
finite-to-one maps between nonconjugate irreducible shifts of finite
type which we know how to construct in significant generality
\cite{Ashley1991, Ashley1993,LindMarcus1995,Kitchens1998,Boyle2005}.
The e-resolving maps, and the maps topologically equivalent to them
(the {\em continuing} maps), are similarly the Markovian maps we
know how to construct in significant generality
\cite{BoyleTuncel1984}.
{ If $\Omega_A, \Omega_B$ are mixing shifts
of finite type with $h(\Omega_A)>h(\Omega_B)$ and there exists any
factor map from $\Omega_A$ to $\Omega_B$ (as there will given a
trivially necessary condition), then there will exist infinitely
many continuing (hence Markovian) factor maps from $\Omega_A$ to
$\Omega_B$. However, the most obvious hope, that the factor map send
the maximal entropy measure of $\Omega_A$ to that of $\Omega_B$, can
rarely be realized. Given $\Omega_A$, there are only finitely many
possible values of topological entropy for $\Omega_B$ for which such
a map can exist \cite{BoyleTuncel1984}. }

\subsection{All factor maps lift 1-1 a.e. to Markovian
maps}\label{sec_1-1aemarkovian}

Here ``all factor maps'' means ``all factor maps between irreducible
sofic {subshifts}''. Factor maps between irreducible SFTs need not
be Markovian, but  they are in the following strong sense close to
being Markovian, even if the subshifts $X$ and $Y$ are only sofic.

\begin{thm}\cite{Boyle2005}\label{th_soficlifttomarkovian}
Suppose $\pi : X\to Y$ is a factor map of irreducible sofic
subshifts.
 Then there are irreducible SFT's
$\Omega_A , \Omega_B$ and
a commuting diagram of factor maps
 \be
\label{soficdiagram}
 \begin{CD}
\Omega_A & @>\ \gamma  >> & \Omega_B \\
 @V \alpha VV & & @VV \beta  V\\
 X & @>>\ \pi \ > & Y
 \end{CD}
 \en
such that $\alpha, \beta$ are degree 1 right resolving and $\gamma$
is e-resolving. In particular, $\gamma$ is Markovian. If $Y$ is
SFT, then the composition $\beta \gamma$ is also Markovian.
\end{thm}

The Markovian claims in Theorem \ref{th_soficlifttomarkovian} hold
because finite-to-one maps are Markovian (Proposition
\ref{prop_finitetoone}), e-resolving maps are Markovian (Proposition
\ref{resolvingprop}), and a composition of Markovian maps is
Markovian. In the case
when $\pi$ is degree 1 between irreducible SFTs, the ``Putnam
diagram'' (\ref{soficdiagram})
 is a special case of Putnam's work
 in \cite{Putnam2005}, which was the stimulus for
\cite{Boyle2005}.

\subsection{Every factor map between SFT's is hidden
Markovian}\label{sec_factorhiddenmarkovian}

A factor map $\pi: \Omega_A \to \Omega_B$ is   Markovian if some
(and therefore every) Markov measure on $\Omega_B$ lifts to a Markov
measure on $\Omega_A$. There exist factor maps between irreducible
SFTs which are not Markovian. In this section we will show in
contrast that all factor maps between irreducible SFTs (and more
generally  between irreducible sofic subshifts) are {\em hidden
Markovian}: every sofic (i.e., hidden Markov) measure lifts to a
sofic measure. The terms Markov measure and sofic measure continue
to include the requirement of full topological support.

\begin{thm}\label{sofician}
Let  $\pi: X \to Y$ be a factor map between irreducible
 sofic subshifts and suppose that $\nu$
is a sofic  measure on Y.
Then $\nu$ lifts to a sofic
measure $\mu$ on X. Moreover, $\mu$ can be chosen to satisfy
$\textnormal{degree}(\mu ) \leq \textnormal{degree}(\nu )$.
\end{thm}
\begin{proof}
We consider two cases.

Case I: $\nu$ is a Markov measure on $Y$.
Consider the
Putnam diagram (\ref{soficdiagram}) associated to $\pi$ in
Theorem \ref{th_soficlifttomarkovian}.
The measure $\nu$ lifts to a Markov measure $\mu^*$ on
$\Omega_A$.
Set $\mu = \alpha \mu^*$.
Then $\pi \mu = \nu$, and
$\text{degree}(\mu )= 1 \leq
\text{degree}(\nu )$.

Case II: $\nu$ is a degree $n$ sofic measure on $Y$. (Possibly
$n=\infty$.) Then there are an irreducible SFT
 $\Omega_C$ with a Markov measure
$\mu'$ and a degree $n$ factor map $g: \Omega_C \to Y$ which sends
$\mu'$ to $\nu$. By Lemma  \ref{surjcomponent}
 below,
there exist another irreducible SFT $\Omega_F$ and factor maps
$\widetilde g$ and $\widetilde{\pi}$ with
 $\textnormal{degree}(\widetilde g) \leq
\textnormal{degree}(g)$ such that the following diagram commutes:
 \be
 \begin{CD}
\Omega_F & @>\ \widetilde{\pi}\ >> & \Omega_C \\
 @V \widetilde g VV & & @VV g V\\
X & @>>\ \pi \ > & Y
 \end{CD}
 \en
Apply Case I to
$\widetilde{\pi}$ to get a degree 1 sofic measure $\nu^*$ on
$\Omega_F $
which $\widetilde{\pi}$ sends to $\mu'$. Then
$\widetilde g(\nu^*)$ is a sofic
measure of  degree at most
$n$  which $\pi$ sends to $\nu$.
\end{proof}

To complete the proof of Theorem \ref{sofician} by proving Lemma
\ref{surjcomponent}, we must recall some background on magic words.
Suppose
$X=\Omega_A$ is SFT and $\pi :\Omega_A\to Y$ is a 1-block factor
map. Any $X$-word $v$ is mapped to a $Y$-word $\pi v$ of equal
length. Given
{ a $Y$-word $w=w[1,n]$}
 and an integer $i$ in
$[1,n]$, set $d(w,i) = |\{w_i': \pi w' =w\}|$. As in
\cite{Boyle2005}, the {\em resolving degree} $\delta (\pi )$ of
$\pi$ is defined as the minimum of $d(w,i)$ over all allowed $w,i$,
and $w$ is a {\em magic word} for $\pi$ if for some $i$,
$d(w,i)=\delta (\pi )$. (For finite-to-one maps, these are the
standard magic words of symbolic dynamics
\cite{LindMarcus1995,Kitchens1998}; some of their properties are
still useful in the infinite-to-one case. The junior author
confesses  an error: \cite[Theorem 7.1]{Boyle2005} is wrong. The
resolving degree is not in general invariant under topological
conjugacy, in contrast to the finite-to-one case.)

If a magic word has length 1, then it is a {\it magic symbol}. As
remarked in \cite[Lemma 2.4]{Boyle2005}, the argument of
\cite[Proposition 4.3.2]{Kitchens1998} still works in the
infinite-to-one case to show that $\pi$ is topologically equivalent
to a 1-block code from a one step irreducible SFT for which there is
a magic symbol. (Factor maps $\pi,\phi$ are topologically equivalent
if there exist topological conjugacies $\alpha,\beta$ such that
$\alpha \phi \beta =\pi$.)

{
\begin{prop} \label{openimage}
Suppose $X$ is
 SFT; $\pi :X\to Y$ is a 1-block factor map;  $a$
is a magic symbol for $\pi$; $aQa$ is a $Y$-word; and $a'Q'a''$ is
an $X$-word such that $\pi (a'Q'a'') = aQa$. Then the image of the
cylinder $\mathcal C_0[a'Q'a'']$ equals the cylinder $\mathcal
C_0[aQa]$.
\end{prop}
\begin{proof}
Suppose $PaQaR$ is a $Y$-word, with preimage $X$-words
$P^ja^jQ^j(a_*)^jR^j$, say $1\leq j\leq J$, with the 1-block code
acting by erasing $*$ and superscripts.
 Because $a$ is a magic symbol,
there must exist some $j$ such that $a_j=a'$, and there must exist
some $k$ such that $(a_*)^k=a''$. Because $X$ is a 1-step SFT,
$P^ja'Q'a''R^k$ is an $X$-word, and it maps to $PaQaR$. This shows
that the image of
 $\mathcal C_0[a'Q'a']$ is dense in
 $\mathcal C_0[aQa]$ and therefore, by compactness,
equal to it.
\end{proof}
}

\begin{cor} \label{cor_open}
Suppose $\pi:X\to Y$ is a factor map from an irreducible SFT $X$ to
a sofic subshift $Y$. Then there is a residual set of points in $Y$
which lift to doubly transitive points in $X$.
\end{cor}
\begin{proof}
Without loss of generality, we  assume $\pi$ is a 1-block factor
map, $X$ is a 1-step SFT, and there is a magic symbol $a$ for $\pi$.
Let $v_n=a'P_na'$, $n\in \mathbb N$, be a set of $X$-words such that
every $X$-word occurs as a subset of some $P_n$ and $a'$ is a symbol
sent to $a$. The set $E_n$ of points in $X$ which see the words
$v_1,v_2,\dots v_n$ both in the future and in the past is a dense
open subset of $X$. It follows from Proposition \ref{openimage} that
each $\pi E_n$ is open. For every $n$,  $E_n$ contains $E_{n+1}$, so
$\pi (\cap_n E_n)=\cap_n \pi E_n$. Thus the set $\cap_n E_n$ of
doubly transitive points in $X$ maps to a residual subset of $Y$.
\end{proof}

We do not know whether in  Corollary \ref{cor_open} every doubly
transitive point of $Y$ must lift to a doubly transitive point of
$X$.

\begin{lem} \label{surjcomponent}
Suppose $\alpha: X\to Z$ and $\beta :Y\to Z$ are  factor maps of
irreducible sofic subshifts. Then there is an irreducible SFT $W$
with factor maps $\widetilde {\alpha}$ and $\widetilde{\beta}$ such
that $\textnormal{degree}(\widetilde{\beta} ) \leq
\textnormal{degree}(\beta )$ and the following diagram commutes.
 \be
 \begin{CD}
W & @>\ \widetilde{\alpha}\ >> & Y \\
 @V \widetilde{\beta} VV & & @VV \beta V\\
X & @>>\ \alpha \ > & Z
 \end{CD}
 \en
\end{lem}

\begin{proof}
First, suppose $X$ and $Y$ are SFT. The intersection of any two
residual sets in $Z$ is nonempty, so by
 Corollary \ref{cor_open} we may find $x$ and $y$, doubly
transitive in $X$ and $Y$ respectively, such that
$\alpha x = \beta y$.
Let $\Omega_F$ be the irreducible component of the
 fiber product
 {
 $\{ (u,v) \in X \times Y : \alpha x = \beta y\}$
 }
 built from $\alpha$
and $\beta$ to which the point $(x,y)$ is forward asymptotic, and
let $\widetilde {\beta},\widetilde{\alpha}$ be restrictions to
$\Omega_F$ of the coordinate projections. These restrictions must be
surjective. Note that $\textnormal{degree}(\widetilde{\beta} ) \leq
\textnormal{degree}(\beta )$.

If $X$ and $Y$ are not necessarily SFT, then there are degree 1
factor maps from irreducible SFT's, $\rho_1 :\Omega_A \to X$ and
$\rho_2 :\Omega_B \to Y$, and we can apply the first case to find
$\widetilde{\alpha \rho_1}$ and $\widetilde {\beta \rho_2}$  in the
diagram with respect to the pair $\alpha \rho_1, \beta \rho_2$. Now
for
 $\widetilde {\alpha}$ and $ \widetilde{\beta}$ we
use the maps
$\rho_1 \widetilde{\alpha \rho_1}$ and
$\rho_2 \widetilde {\beta \rho_2}$.
\end{proof}

\begin{ack*}
 This article arose from the October 2007 workshop ``Entropy of Hidden
 Markov Processes and Connections to Dynamical Systems''
at the Banff International Research Station, and we thank BIRS,
PIMS, and MSRI for hospitality and support. We thank Jean-Ren\'e
Chazottes, Masakazu Nasu, Sujin Shin, Peter Walters and Yuki Yayama
for very helpful comments. We are especially grateful to Uijin Jung
and the two referees for extremely thorough comments and
corrections.
 Both authors
thank the  Departamento de Ingen{i}er\'{i}a Matem\'{a}tica, Center
for Mathematical Modeling, of the University of Chile and the
CMM-Basal Project, and the second author also the Universit\'e
Pierre et Marie Curie (University of Paris 6) and Queen Mary
University of London, for hospitality and support
 during the preparation of this article.
 Much of Section \ref{sec_ident} is drawn from lectures given
 by the second author in a graduate course at the University of North
 Carolina, and we thank  the students who wrote up the notes:
 Rika Hagihara, Jessica Hubbs, Nathan Pennington, and Yuki Yayama.
\end{ack*}

\bibliographystyle{amsplain}
\bibliography{MBKP04Jan2010}

\end{document}